\documentclass[11pt]{article}
\usepackage[ansinew]{inputenc}
\usepackage{graphicx}
\usepackage{color}

\usepackage{enumerate,latexsym}
\usepackage{latexsym}
\usepackage{amsmath,amssymb}
\usepackage{graphicx}
\newfont{\bb}{msbm10 at 11pt}
\newfont{\bbsmall}{msbm8 at 8pt}
\def\rth{\mathbb{R}^3}
\def\R{\mathbb{R}}
\def\B{\mathbb{B}}
\def\N{\mathbb{N}}

\def\C{\mathbb{C}}

\def\Hip{\mathbb{H}}

\def\D{\mathbb{D}}

\def\esf{\mathbb{S}}

\newcommand{\ben}{\begin{enumerate}}
\newcommand{\bit}{\begin{itemize}}
\newcommand{\een}{\end{enumerate}}
\newcommand{\eit}{\end{itemize}}
\newcommand{\wh}{\widehat}

\newcommand{\wt}{\widetilde}
\newcommand{\HH}{\mbox{\bb H}}

\def\a{{\alpha}}
\def\lc{{\cal L}}
\def\t{{\theta}}

\def\g{{\gamma}}
\def\G{{\Gamma}}
\def\l{{\lambda}}
\def\L{{\Lambda}}
\def\de{{\delta}}
\def\be{{\beta}}
\def\ve{{\varepsilon}}

\def\centerbmp#1#2#3{\vskip#2\relax\centerline{\hbox to#1{\special
    {bmp:#3 x=#1, y=#2}\hfil}}}

\newtheorem{theorem}{Theorem}[section]
\newtheorem{lemma}[theorem]{Lemma}
\newtheorem{proposition}[theorem]{Proposition}

\newtheorem{remark}[theorem]{Remark}
\newtheorem{corollary}[theorem]{Corollary}
\newtheorem{definition}[theorem]{Definition}
\newtheorem{conjecture}[theorem]{Conjecture}
\newtheorem{assertion}[theorem]{Assertion}

\newenvironment{proof}{\smallskip\noindent{\it Proof.}\hskip \labelsep}
{\hfill\penalty10000\raisebox{-.09em}{$\Box$}\par\medskip}

\begin{document}
\begin{title}
{Local removable singularity theorems for minimal laminations}
\end{title}
\vskip .2in

\begin{author}
{William H. Meeks III\thanks{This material is based upon
   work for the NSF under Award No. DMS -
   1004003.  Any opinions, findings, and conclusions or recommendations
   expressed in this publication are those of the authors and do not
   necessarily reflect the views of the NSF.}
   \and Joaqu\'\i n P\' erez
\and Antonio Ros\thanks{The second and third authors were supported in part
by MEC/FEDER grants no. MTM2007-61775 and MTM2011-22547, and
Regional J. Andaluc\'\i a grant no. P06-FQM-01642.}}
\end{author}
\maketitle
\begin{abstract}
In this paper we prove a {\it local removable singularity theorem}
for certain minimal laminations with isolated singularities in a
Riemannian three-manifold. This removable singularity theorem
is the key result used in our proof that a
complete, embedded minimal surface in $\R^3$ with quadratic decay of
curvature has finite total curvature.
\vspace{.3cm}


\end{abstract}
\noindent

\section{Introduction.}

There is a vast literature concerning the (local) analysis of
mathematical objects around an isolated singularity. In this paper we shall be
concerned with the local behavior of an embedded minimal surface, or more generally,
of a minimal lamination ${\cal L}$ of a punctured ball in a Riemannian three-manifold,
with the central question of when such an ${\cal L}$ extends as lamination across the
puncture. We will characterize the removability of the
possible singularity of the closure of ${\cal L}$ at the puncture in
terms of the growth of the norm of the second fundamental form of
the leaves of ${\cal L}$ when approaching the puncture; see the
Local Removable Singularity Theorem (Theorem~\ref{tt2}) below
for this characterization.

Before stating the Local Removable Singularity Theorem, we set some specific
notation to be used throughout the paper. Given a three-manifold $N$ and
a point $p\in N$, we denote by $d_N$ the distance function in $N$ to $p$
and by $B_N(p,r)$ the open metric ball of center $p$ and radius $r>0$.
For a lamination  $\lc$ of $N$ and a leaf $L$ of
${\cal L}$, we denote by $|\sigma _{L}|$
the norm of the second fundamental form of $L$. Since
leaves of ${\cal L}$ do not intersect, it makes sense
to consider the norm of the second fundamental form as a
function defined on the union of the leaves of
${\cal L}$, which we denote by  $|\sigma _{{\cal L}}|$.
In the case $N=\R^3$, we use the notation $\B (p,r)=B _{\R^3}(p,r)$ and
$\B (r)=\B (\vec{0},r)$.  The boundary and closure of $\B (r)$ will be
respectively denoted by $\partial \B (r)=\esf^2(r)$ and
$\overline{\B }(r)$. $\esf^1(r)$ represents the circle
$\{ (x_1,x_2)\ | \ x_1^2+x_2^2=r^2\} \subset \R^2$. Furthermore,
$R\colon\R^3\to \R $ will stand for the distance function to
the origin $\vec{0}\in \R^3$. Finally, for a surface
$M\subset \R^3$, $K_M$ denotes its Gaussian curvature function.

\begin{theorem} [Local Removable Singularity Theorem]
\label{tt2}
A minimal lamination $\lc$ of a punctured ball $B_N(p,r)-\{ p\} $
in a Riemannian three-manifold $N$ extends to a minimal lamination
of $B_N(p,r)$ if and only if there exists a positive constant $C$
such that $|\sigma _{{\cal L}}|\, d_N(p,\cdot )\leq C$ in some subball.
In particular under this hypothesis,
\begin{enumerate}
\item The curvature of  ${\cal L}$ is bounded in a neighborhood of $p$.
\item If ${\cal L}$ consists of a single leaf $M\subset B_N(p,r)-\{ p\} $
which is a properly embedded minimal surface, then
$M$ extends smoothly through $p$.
\end{enumerate}
\end{theorem}

We remark that the natural generalization of the above local
removable singularity theorem
fails badly for codimension-one minimal laminations of $\R^n$, for $n
= 2 $ and for $n > 3$. In the case $n= 2$, consider the cone ${\cal
C}$ over any two non-antipodal points on the unit circle; ${\cal C}$
consists of two infinite rays making an acute angle at the origin.
The punctured cone ${\cal C} - \{ \vec{0} \}$ is totally geodesic
and so, the norm of the second fundamental form of ${\cal
C}-\{\vec{0}\}$ is zero but ${\cal C}$ is not a smooth lamination at
the origin. In the case $n\geq 4$, let ${\cal C}$ denote the cone over
any embedded, compact minimal hypersurface $\Sigma $ in
$\esf^{n-1}$ which is not an equator.
Since the norm of the second
fundamental form of $\Sigma $ is bounded, then the norm of
the second fundamental form of ${\cal C} - \{ \vec{0} \}$ multiplied
by the distance function to the origin is also a bounded function
on ${\cal C} - \{ \vec{0} \}$.
These examples demonstrate that Theorem~\ref{tt2}
is precisely an ambiently three-dimensional result.

Theorem~\ref{tt2} is related to previous results by Colding and
Minicozzi, where they obtain quadratic estimates for the area of a
compact embedded minimal surface $\Sigma \subset \B (R)$
with connected boundary $\partial \Sigma \subset \esf ^2(R)$, assuming
a quadratic estimate of its Gaussian curvature and a concentration of the
genus of $\Sigma $ in a smaller ball; see
Theorem~0.5 and Corollary~0.7 in~\cite{cm24}.

An important application of Theorem~\ref{tt2} to the classical
theory of minimal surfaces is to characterize complete embedded
minimal surfaces with quadratic curvature decay. For the statement
of the next theorem, we first recall that a complete Riemannian
surface $M$ has { \em intrinsic quadratic curvature decay constant
$C>0$ with respect to a point $p\in M$}, if the absolute Gaussian
curvature function $|K_M|$ of $M$ satisfies
\[
|K_M(q)|\leq \frac{C}{d_M(p,q)^2}\quad \mbox{for all $q\in M$,}
\]
 where $d_M$ denotes the Riemannian distance
function. Since the intrinsic distance $d_M$ dominates the
ambient extrinsic distance in $\R^3$, we deduce that if a complete Riemannian surface
$M$ in $\R^3$ with $p=\vec{0}\in M$ has intrinsic quadratic curvature
decay constant $C$ with respect to $\vec{0}$, then it also has extrinsic
quadratic decay constant $C$ with respect to the radial distance $R$
to $\vec{0}$, in the sense that $|K_M|R^2 \leq {C}$ on $M$. For this reason, when
we say that a minimal surface in $\rth$ has {\em quadratic decay of
curvature}, we will always refer to curvature decay with respect to
the extrinsic distance $R$ to $\vec{0}$, independently of whether or
not $M$ passes through $\vec{0}$. Note that the property of having
quadratic decay of curvature is scale-invariant, a fact that will be
crucial throughout this paper.

\begin{theorem}[Quadratic Curvature Decay Theorem]
\label{thm1introd}
A complete, embedded minimal surface in $\R^3$
with compact boundary (possibly empty) has quadratic decay of
curvature if and only if it has finite total curvature. In
particular, a complete, connected embedded minimal surface $M\subset
\R^3$ with compact boundary and quadratic decay of curvature is
properly embedded in $\R^3$. Furthermore, if $C$ is the maximum of
the logarithmic growths of the ends of $M$, then
\[
\lim _{R\to \infty }\sup _{M-{\mbox{\bbsmall B}}(R)}|K_M|R^4=C^2.
\]
\end{theorem}

The Local Removable Singularity Theorem in this paper
will be applied in forthcoming papers to obtain
the following results.
\begin{enumerate}
\item A dynamics type result for the space of all limits of
a given non-flat, properly embedded minimal surface in $\R^3$
under divergent sequences of dilations in~\cite{mpr20}.

\item A blow-up technique on the scale of non-trivial topology for describing the
local structure of a complete embedded minimal surface with injectivity radius zero
in a homogeneously regular Riemannian three-manifold in~\cite{mpr14}.

\item Global structure theorems for certain possibly singular
minimal laminations of $\rth$ in \cite{mpr11}.

\item Bounds for the number
of ends and for the index of stability of all
complete, embedded minimal surfaces in $\rth$ with finite
topology, more than one end and having fixed genus in \cite{mpr8}.

\item Calabi-Yau type results. For example, a complete
embedded minimal surface in $\R^3$ of finite genus is
properly embedded if and only if it has a countable
number of ends, see~\cite{mpr9};
    this result generalizes a theorem by Colding
    and Minicozzi~\cite{cm35}, who proved the sufficient
    implication in the case that the number of ends is finite.

\item We will extend in~\cite{mpr21} the Local Removable Singularity Theorem
from the minimal case (that is the leaves of ${\cal L}$ in Theorem~\ref{tt2}
are minimal surfaces) to the case of an $H$-lamination (all the leaves of
${\cal L}$ have the same constant mean curvature $H\in \R $). In this generalization,
we will even allow the leaves of ${\cal L}$ to intersect tangentially. This
extended Local Removable Singularity Theorem is the key tool for the
classification of all CMC (constant mean curvature) foliations of $\R^3$ and $\esf ^3$ with a closed countable
number of singularities, where by a CMC foliation we mean that the leaves of
the foliation have constant mean curvature, possibly varying from leaf to leaf.
We point out that the statement of this classification
was announced in Theorem 6.8 of~\cite{mpr19}; there we only gave there a proof
in the particular case of $N=\R^3$ and ${\cal S}$ is finite, and referred the
reader to an earlier version of the present manuscript for the proof in the general case. Nevertheless,
for the sake of simplicity we will only deal here with minimal surfaces and laminations,
and postpone both the generalization of Theorem~\ref{tt2} to the case
of $H$-laminations and the proof of the classification of CMC
foliations of $\R^3-{\cal S}$ and $\esf^3-{\cal S}$, ${\cal S}$ being a closed countable set,
 to the paper~\cite{mpr21}.

 \end{enumerate}

The paper is organized as follows. In Section~\ref{secex} we present
the basic definitions and examples of minimal laminations in $\R^3$
and in Riemannian three-manifolds. Section~\ref{sec3} is devoted to
prove a key lemma on stable minimal surfaces which are complete outside
a point in $\R^3$. In Section~\ref{sec4} we apply results by
Colding and Minicozzi~\cite{cm23} to demonstrate a key technical
result about minimal laminations of $\R^3-\{ \vec{0}\} $ with
quadratic decay of curvature, which will be used in
Section~\ref{sectionproofremov} when proving the
Local Removable Singularity Theorem. In Section~\ref{sec6} we apply
the Local Removable Singularity Theorem to
characterize complete embedded minimal surfaces with quadratic decay
of curvature in $\R^3$ as being surfaces with finite total
curvature. We will give some applications to the case
of minimal surfaces or minimal laminations with countably many singularities
in Section~\ref{sec10}. We finish the paper with an appendix
that contains a second
proof of the main technical result in Section~\ref{sec4}, which does not depend on
the results of Colding and Minicozzi in~\cite{cm23}.

The authors would like to thank David Hoffman for helpful
suggestions and Brian White for explaining to us classical work
of Allard and Almgren concerning limit tangent cones for surfaces
with bounded mean curvature outside an isolated singular point.

\section{Basic definitions and some examples.} \label{secex}

\begin{definition}\label{def-limset} {\rm
Let $M$ be a complete, embedded surface in a three-manifold $N$. A
point $p\in N$ is a {\it limit point} of $M$ if there exists a
sequence $\{p_n\}_n\subset M$ which diverges to infinity in $M$ with
respect to the intrinsic Riemannian topology on $M$ but converges in
$N$ to $p$ as $n\to \infty$. Let lim$(M)$ denote the set of all limit
points of $M$ in $N$; we call this set the {\it limit set of $M$}.
In particular, lim$(M)$ is a closed subset of $N$ and $\overline{M} -M
\subset \lim (M)$, where $\overline{M}$ denotes the closure of~$M$.}
\end{definition}

\begin{definition}
\label{deflamination}
{\rm
 A {\it codimension one
lamination} of a Riemannian three-manifold $N$ is the union of a
collection of pairwise disjoint, connected, injectively immersed
surfaces, with a certain local product structure. More precisely, it
is a pair $({\mathcal L},{\mathcal A})$ satisfying:
\begin{enumerate}
\item ${\mathcal L}$ is a closed subset of $N$;
\item ${\mathcal A}=\{ \varphi _{\be }\colon \D \times (0,1)\to
U_{\be }\} _{\be }$ is an atlas of coordinate charts of $N$ (here
$\D $ is the open unit disk in $\R^2$, $(0,1)$ is the open unit
interval and $U_{\be }$ is an open subset of $N$); note that
although $N$ is assumed to be smooth, we
only require that the regularity of the atlas (i.e. that of
its change of coordinates) is of class $C^0$, i.e. ${\cal A}$
is an atlas with respect to the topological structure of $N$.
\item For each $\be $, there exists a closed subset $C_{\be }$ of
$(0,1)$ such that $\varphi _{\be }^{-1}(U_{\be }\cap {\mathcal L})=\D \times
C_{\be}$.
\end{enumerate}

We will simply denote laminations by ${\mathcal L}$, omitting the
charts $\varphi _{\be }$ in ${\mathcal A}$. A lamination ${\mathcal
L}$ is said to be a {\it foliation of $N$} if ${\mathcal L}=N$.
Every lamination ${\mathcal L}$ naturally decomposes into a
collection of disjoint connected topological surfaces (locally given by $\varphi
_{\be }(\D \times \{ t\} )$, $t\in C_{\be }$, with the notation
above), called the {\it leaves} of ${\mathcal L}$.
Note that if $\Delta
\subset {\cal L}$ is any collection of leaves of ${\cal L}$, then
the closure of the union of these leaves has the structure of a
lamination within ${\cal L}$, which we will call a {\it
sublamination.} The notion of limit point of a complete embedded
surface (Definition~\ref{def-limset}) can be extended to the case of
a lamination $\mathcal{L}$ of a three-manifold $N$ as follows. A point
$p\in \mathcal{L}$ is a {\it limit point} if there exists a coordinate chart
$\varphi _{\beta }\colon \D \times (0,1)\to U_{\be }$ as in Definition~\ref{deflamination}
such that $p\in U_{\be }$ and $\varphi _{\be }^{-1}(p)=(x,t)$ with
$t$ belonging to the accumulation set of $C_{\be }$.
It is easy to show that if $p$ is a limit point of
a lamination ${\cal L}$, then the leaf $L$ of ${\cal L}$ passing through $p$
consists entirely of limit points of ${\cal L}$; see Footnote~\ref{footnote}.
In this case, $L$ is called a {\it limit leaf} of ${\cal L}$.
 }
\end{definition}

A lamination ${\cal L}$ of $N$ is said to be a {\it minimal lamination}
if each of its leaves is a smooth surface with zero mean curvature.
In this case, the function $|\sigma _{\cal L}|$ that associates to each
point $p$ of ${\cal L}$ the norm of the second fundamental form of the unique leaf
of $\mathcal{L}$ passing through $p$, makes sense on $\mathcal{L}$.
 A natural question to ask is whether or not the function
$|\sigma _{\cal L}|$ is locally bounded for any
minimal lamination ${\cal L}$ in a Riemannian
three-manifold $N$. Concerning this question, we
observe  that the 1-sided curvature estimates for minimal disks by Colding and
Minicozzi~\cite{cm23,cm35} imply that $|\sigma _{\cal L}|$
is locally bounded (to prove this, one only has to deal
with limit leaves, where the 1-sided curvature estimates apply).
Another important observation is that given a
sequence of minimal laminations ${\cal L}_n$ of $N$ with uniformly bounded
second fundamental forms
on compact subdomains of $N$, a subsequence of the ${\cal L}_n$ converges
to a minimal lamination of $N$; see Proposition~B1 in~\cite{cm23}.


\subsection{Minimal laminations with isolated singularities.}

We first construct examples in the closed unit ball of $\R^3$
centered the origin, with the origin as the unique non-removable
singularity.  We then show how these examples lead to related
singular minimal laminations in the hyperbolic space ${ \HH }^3$.

\begin{description}
\item[{\sc Example I.}] {\it Catenoid type laminations.}
Consider the sequence of horizontal circles
$C_n=\esf^2(1)\cap \{ x_3=\frac{1}{n}\} $, $n\geq 2$.
   Note that each pair
$C_{2k},C_{2k+1}$ bounds a compact unstable catenoid $M(k)\subset \overline{\B }(1)$.
Clearly, $M(k)\cap M(k')=\mbox{\O }$ if $k\neq k'$. The sequence $\{
M(k)\} _k$ converges with multiplicity two outside of the origin
$\vec{0}$ to the closed horizontal disk $\overline{\D }$ of radius 1
centered at $\vec{0}$. Thus, $\{ M(k)\} _k\cup \{ \overline{\D }-\{
\vec{0}\} \} $ is a minimal lamination of $\overline{\B }(1)-
\{ \vec{0}\} $ which does not extend through the origin; see
Figure~\ref{example1} left.
\begin{figure}
\begin{center}
\includegraphics[width=13.1cm]{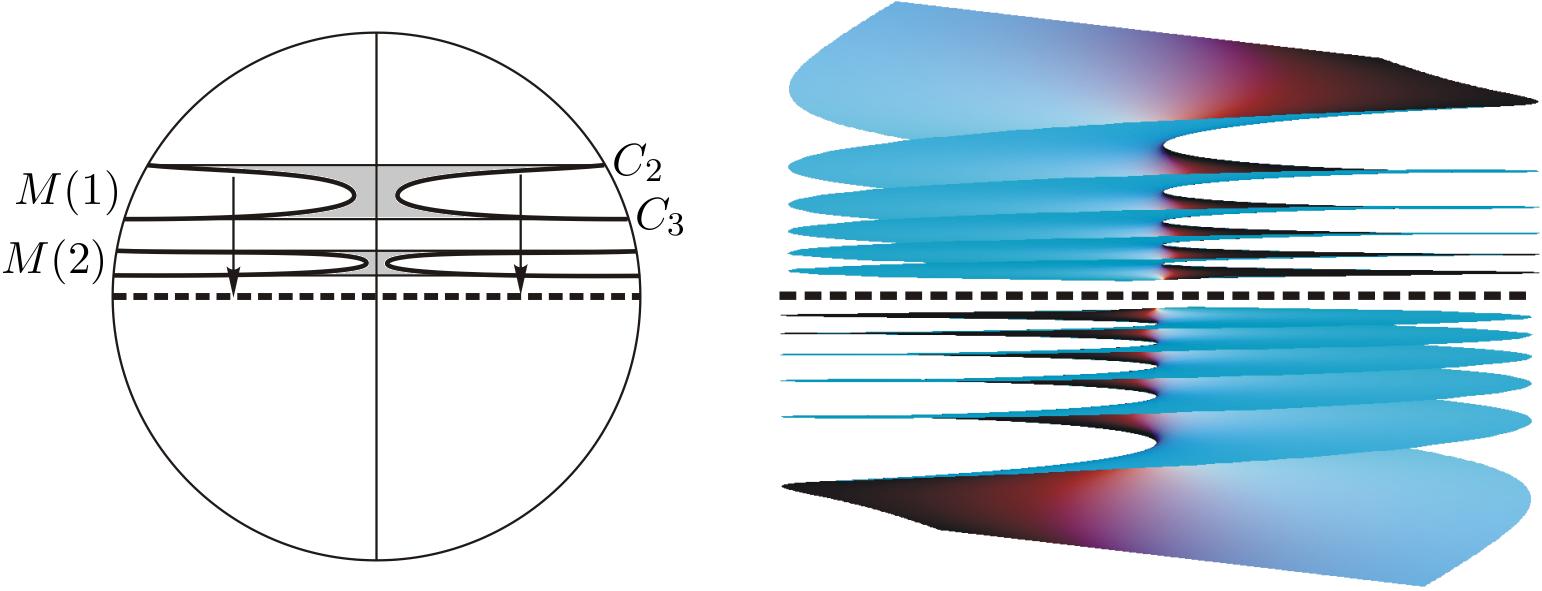}
\caption{Left: A catenoid type lamination. Right: A Colding-Minicozzi
type lamination in a cylinder.}
\label{example1}
\end{center}
\end{figure}

\item[{\sc Example II.}] {\it Colding-Minicozzi examples.}
In their paper~\cite{cm28}, Colding and Minicozzi constructed a
sequence of compact, embedded minimal disks $D_n\subset \overline{\B }(1)$
with boundaries in $\esf ^2(1)$, that converges to a singular
minimal lamination $\overline{{\cal L}}$ of $\overline{\B }(1)$
with an isolated singularity at
$\vec{0}$. The related lamination ${\cal L}$ of $\overline{\B}
(1) - \{ \vec{0} \}$ consists of a unique limit leaf which
is  the punctured closed disk $\overline{\D }-\{ \vec{0}\} $,
together with two non-proper leaves that spiral into $\overline{\D
}-\{ \vec{0}\} $ from opposite sides; see Figure~\ref{example1} right.

Consider the exhaustion of $\Hip ^3$ (identified with $\B
(1)$ through the Poincar\'e model) by hyperbolic geodesic balls of
hyperbolic radius $n\in \N $ centered
at the origin, together with compact minimal disks with boundaries
on the boundaries of these balls, similar to the compact
Colding-Minicozzi disks. We conjecture that these examples produce a
similar limit lamination of $\Hip ^3-\{ \vec{0}\} $ with three
leaves, one which is totally geodesic and the other two which are
not proper and that spiral into the first one. We remark that one of
the main results of Colding-Minicozzi theory (Theorem 0.1
in~\cite{cm23}) insures that such an example cannot be constructed
in $\R^3$.

\item[{\sc Example III.}] {\it Catenoid type examples in
$\Hip ^3$ and in $\Hip ^2\times \R $.}
As in example I, consider the circles $C_n=\esf ^n(1)\cap \{ x_3=\frac{1}{n}\} $,
where $\esf^2(1)$ is now viewed as the boundary at infinity of $\Hip ^3$.
Then each pair of circles $C_{2k},C_{2k+1}$ is the asymptotic boundary of
a properly embedded annular minimal unstable surface $M(k)$,
which is a surface of revolution called a catenoid (see e.g.,
Lemma~3.5 in Gomes~\cite{gom1}). The sequence $\{ M(k)\} _k$ converges with
   multiplicity two outside of $\vec{0}$ to the horizontal
totally geodesic subspace $\D $ at height zero. Thus,
$\{ M(k)\} _k\cup \{ \D -\{ \vec{0}\} \} $ is
a minimal lamination of $\Hip ^3-\{ \vec{0}\} $, which does
not extend through the origin. A similar catenoidal construction
can be done in $\Hip ^2\times \R $, where we consider $\Hip ^2$
in the disk model of the hyperbolic plane, using the minimal catenoids
constructed in Theorem~1 of~\cite{ner2}.
Note that the Half-space Theorem~\cite{hm10} excludes this
type of singular minimal lamination in $\R^3$.
\end{description}

\subsection{Minimal laminations with limit leaves.}

\begin{description}
\item[{\sc Example IV.}] {\it Simply-connected bridged examples.}
Coming back to the Euclidean closed unit ball $\overline{\B }(1)$,
consider the sequence of horizontal
closed disks $\D _n=\overline{\B }(1)\cap \{ x_3=\frac{1}{n}\} $, $n\geq 2 $.
Connect each pair $\D _n,\D _{n+1}$ by a thin, almost
   vertical minimal  bridge (in opposite sides for consecutive
disks, as in Figure~\ref{example4} left), and perturb slightly
this non-minimal surface to obtain an embedded, stable minimal
surface with boundary in
$\overline{\B }(1)$ (this is possible by the
{\it bridge principle}~\cite{my2,wh11,wh12}).
We denote by $M$ the intersection of
   this surface with $\B (1)$. Then, the closure of
$M$ in $\B (1)$ is a minimal lamination of
$\B (1)$ with two leaves, both being stable, one of which is
$\D$ (this is a limit leaf) and the other one is not flat and
not proper.
\begin{figure}
\begin{center}
\includegraphics[width=12.38cm,height=4.5cm]{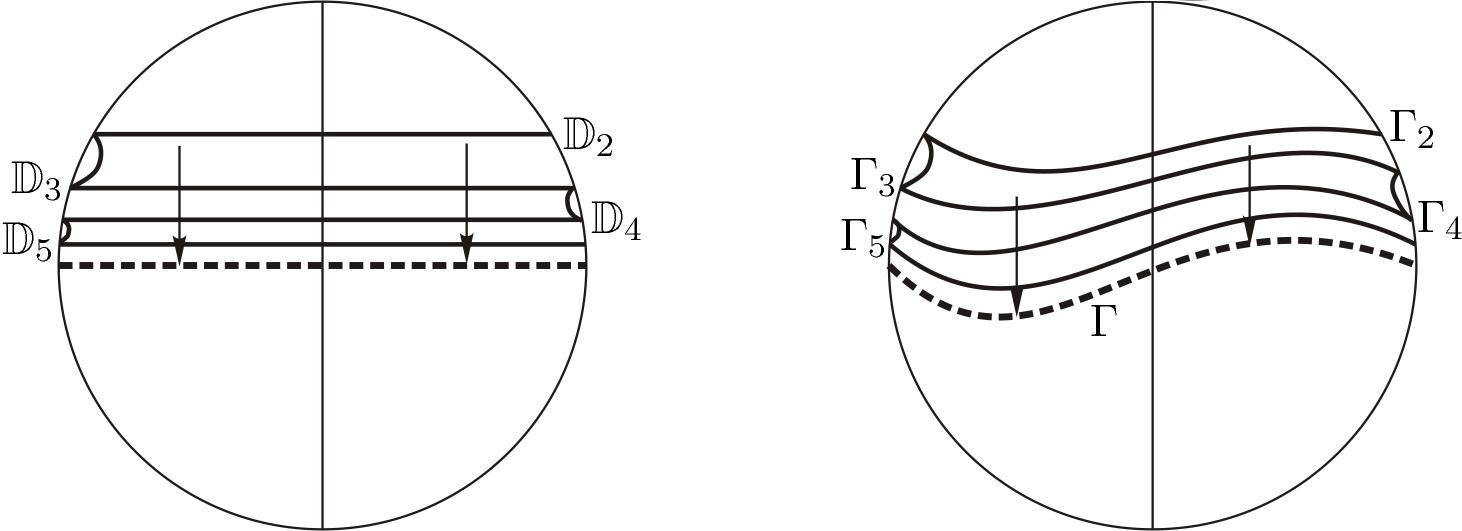}
\caption{Left: Almost flat minimal disks joined by small bridges.
Right: A similar example with a non-flat limit leaf.}
\label{example4}
\end{center}
\end{figure}

A similar example with a non-flat limit leaf can be
   constructed by exchanging the horizontal circles
by suitable curves in $\esf^2(1)$. Consider a non-planar
   smooth Jordan curve $\G \subset \esf^2(1)$ which
   admits a one-to-one projection onto a convex planar
curve in a plane $\Pi $.
Let $\G _n$ be a sequence of smooth Jordan curves in
   $\esf^2(1)$ converging to
$\G $, so that each $\G _n$ also projects injectively onto a convex
planar curve in $\Pi $ and $\{ \G _n\} _n\cup \{ \G \} $ is a
lamination on $\esf^2(1)$. An elementary application of the maximum principle
implies that each of the $\G _n$ is the boundary of a
unique compact minimal surface $\overline{M_n}$, which is a graph over its
projection to $\Pi $. Now join slight perturbations of the
$\overline{M_n}$ by thin bridges as in the preceding paragraph, to
obtain a simply-connected minimal surface in the closed unit ball.
Let $M$ be the intersection of this surface with $\B (1)$.
Then, the closure of $M$ in $\B (1)$ is a minimal lamination
of $\B (1)$
   with two leaves, both being non-flat and stable,
and exactly one of them is properly embedded in
$\B (1)$ and is a limit leaf; see Figure~\ref{example4} right.

\item[{\sc Example V.}] {\it Simply-connected bridged examples in $\Hip ^3$.}
As in the previous subsection, the minimal laminations in example IV give rise to
minimal laminations of $\Hip ^3$ consisting
of two stable, complete, simply connected minimal surfaces, one of
which is proper and the other one which is not proper in the space,
and either one is not totally geodesic or both of them are not
totally geodesic, depending on the choice of the Euclidean model
surface in Figure~\ref{example4} (for this existence, one can use
Anderson~\cite{an4} to create the corresponding minimal disks $\overline{M_n}$
as in Example IV, and then use the bridge principle at infinity as described in
Coskunuzer~\cite{cosk1}). In this case, the proper leaf is
the unique limit leaf of the minimal lamination. More generally,
Theorem~13 in~\cite{mr13} implies that the closure of any complete,
embedded minimal surface of finite topology in $\Hip ^3$
has the structure of a minimal lamination.
\end{description}

\section{Stable minimal surfaces which are complete outside of a point.}
\label{sec3}

\begin{definition}
\label{defcompl} {\rm A surface $M\subset \R^3-\{ \vec{0}\} $ is
{\it complete outside the origin,} if every divergent path in $M$ of
finite length has as limit point the origin. }
\end{definition}

If $M$ is a complete, stable, orientable minimal surface in $\R^3$,
then $M$ is a plane~\cite{cp1,fs1,po1}. Our goal in this section is to
extend this result to the case where $M$ is complete outside the
origin.

\begin{remark}
{\rm
\begin{description}
\item[]
\item[{\rm 1.}] In Sections 4, 5 and 6 we shall study complete, embedded minimal surfaces
$M\subset \R^3$ with quadratic decay of curvature. Our approach is
to produce from $M$, via a sequence of homothetic rescalings, a
minimal lamination ${\cal L}$ of $\R^3-\{ \vec{0}\} $ with a limit
leaf $L$. Since $L$ is a leaf of a minimal lamination of $\R^3-\{
\vec{0}\} $, then $L$ is complete outside the origin. After
possibly passing to its orientable two-sheeted cover and applying
the main theorem of~\cite{mpr18}, we can assume that
$L$ is stable, orientable and complete outside the origin.
The following lemma will then be
used to show that the closure of $L$ is a plane. This planar leaf
$L$ will play a key role in proving that $M$ must have finite total
curvature.

\item[{\rm 2.}] The line of arguments in item 1 of this remark is inspired by
ideas in our previous paper~\cite{mpr4}, where we proved that a
properly embedded minimal surface of finite genus in $\R^3$ cannot
have one limit end. A key lemma in the proof of this result states
that if such a surface $M$ exists, then some sequence of homothetic
shrinkings of $M$ converges to a minimal lamination of $\R^3-\{
\vec{0}\} $. Furthermore, this lamination is contained in a closed
half-space and contains a limit leaf $L$, which is different from
the boundary of the half-space. Since $L$ is a leaf of a minimal
lamination of $\R^3-\{ \vec{0}\} $, then it is complete outside
$\vec{0}$ and as it is a limit leaf, it is stable. We then proved that
the closure $\overline{L}$ of $L$ must be a plane. Using the plane
$\overline{L}$ as a guide for understanding the lamination, we obtained a
contradiction.

\item[{\rm 3.}] The minimal case of Lemma~\ref{lema1} below was found independently by
Colding and Minicozzi~\cite{cm25}.
\end{description}
}
\end{remark}

\begin{lemma} [Stability Lemma] \label{lema1}
Let $L\subset \R^3-\{ \vec{0}\} $ be a stable, immersed
constant mean curvature (orientable if minimal) surface,
which is complete outside the
origin. Then, $\overline{L}$ is a plane.
\end{lemma}
\begin{proof}
We will present a detailed proof in the minimal
case, since this is the version needed in this paper; see
Lemma~6.4 of~\cite{mpr19} for a proof in the general case of
constant mean curvature.
If $\vec{0}\notin \overline{L}$, then $L$ is complete and so, it is
a plane. Assume now that $\vec{0}\in\overline{L}$. Consider the
metric $\widetilde{g}=\frac{1}{R^2}g$ on $L$, where $g$ is the
metric on $L$ induced by the usual inner product $\langle ,\rangle $ of
$\R^3$. Note that if $L$ were a plane through $\vec{0}$, then
$\widetilde{g}$ would be the metric on $L$ of an  infinite cylinder
of radius~1 with ends at $\vec{0}$ and at infinity. Since
$(\R^3-\{ \vec{0}\} ,\widehat{g})$ with
$\widehat{g}=\frac{1}{R^2}\langle ,\rangle $, is isometric to $\esf
^2(1)\times \R $, then $(L,\widetilde{g})\subset (\R^3-\{ \vec{0}\}
,\widehat{g})$ is complete.

We next show how the assumption of stability can be used to prove
that $(L,g)$ is flat. The laplacians and Gauss
curvatures of $g,\widetilde{g}$ are related by the equations
$\widetilde{\Delta }=R^2\Delta $ and $\widetilde{K}=R^2(K_{L}+\Delta
\log R)$. Since $\Delta \log R=\frac{2(1-\| \nabla R\| ^2)}{R^2}\geq
0$, then
\[
-\widetilde{\Delta }+\widetilde{K}=R^2(-\Delta +K_{L}+\Delta \log R)\geq
R^2(-\Delta +K_{L}).
\]
Since $K_{L}\leq 0$ and $(L,g)$ is stable, $-\Delta +K_{L}\geq
-\Delta +2K_{L}\geq 0$, and so, $-\widetilde{\Delta }+\widetilde{K}
\geq 0$ on $(L,\widetilde{g})$. As $\widetilde{g}$ is complete, the
universal covering of $L$ is conformally $\C $ (Fischer-Colbrie and
Schoen~\cite{fs1}). Since $(L,g)$ is stable, there exists a positive
Jacobi function $u$ on $L$. Passing to the universal covering
$\widehat{L}$, we have $\Delta \widehat{u}=2K_{\widehat{L}}\widehat{u}\leq
0$. Thus $\widehat{u}$ is a positive superharmonic on $\C $, and
hence constant. Therefore, $0=\Delta u-2K_{L}u=-2K_{L}u$ on $L$,
which means that $K_L=0$.
\end{proof}

We will need the following two corollaries. The first one follows directly
from Lemma~\ref{lema1} and the fact that the two-sided cover of every limit leaf of a minimal
lamination is stable, see~\cite{mpr18}. For the second corollary, we refer the reader
to Lemma~4.2 in~\cite{mpr20} which implies that (with the
notation of Corollary~\ref{sc2} below) the two-sided cover  $\wh{L}$
of $L$ is stable; hence, Lemma~\ref{lema1} applies to give that $\wh{L}$
(and so $L$) is flat.

\begin{corollary}
\label{sc1}
If $L$ is a limit leaf of a minimal lamination of $\rth - \{ \vec{0} \}$, then
$\overline{L}$ is a plane.
\end{corollary}

\begin{corollary}
\label{sc2}
Let $\lc$ be a minimal lamination of $\rth$ (resp. of $\R^3-\{ 0\} $)
which is a limit of embedded minimal surfaces $M_{n}$ with uniformly
 bounded second fundamental form on compact sets in $\rth$ (resp. of $\R^3-\{ 0\} $). Let $L$ be a leaf of ${\cal
L}$ which is not a limit leaf of $\lc$, such that the
multiplicity of the limit $\{ M_{n} \}_n\to L$ is greater than one.
Then, $L$ (resp. $\overline{L}$) is a plane.
\end{corollary}

\section{Minimal laminations with quadratic decay of curvature.}
\label{sec4}
In this section we will obtain a preliminary description of any
non-flat minimal lamination ${\cal L}$ of $\R^3-\{ \vec{0}\} $ with
quadratic decay of curvature, see Definition~\ref{def4} below. We
first consider the simpler case where ${\cal L}$ consists of a
properly embedded  minimal surface in $\R^3$. When the decay
constant for its curvature is small, then the topology and geometry of the surface is
simple, as shown in the next lemma.

\begin{lemma}
\label{lemma2} There exists $C\in (0,1)$ such that if $M\subset
\R^3-\B (1)$ is a properly embedded, connected minimal surface with
non-empty boundary $\partial M\subset \esf ^2(1)$ and
$|K_M|R^2\leq C$ on $M$, then $M$ is an annulus which has a planar
or catenoidal end.
\end{lemma}
\begin{proof}
First let $C$ be any positive number less than 1.
Let $f=R^2$ on $M$. Its critical points occur
at those $p\in M$ where $M$ is tangent to $\esf ^2(|p|)$. The
hessian $\nabla ^2f$ at a critical point $p$ is $(\nabla
^2f)_p(v,v)=2\left( |v|^2-\sigma _p(v,v)\langle p,\vec{n}\rangle \right)
$, $v\in T_pM$, where $\sigma $ is the second fundamental form of
$M$  and $\vec{n}$ its Gauss map. Taking $|v|=1$, we  have
$\sigma_p(v,v)\leq |\sigma _p(e_i,e_i)|=\sqrt{|K_M|}(p)$, where
$e_1,e_2$ is an orthonormal basis of principal directions at $p$.
Since $\langle p,\vec{n}\rangle \leq |p|$, we have
\begin{equation}
(\nabla ^2f)_p(v,v)\geq 2\left[ 1-(|K_M|R^2)^{1/2}\right] \geq 2(1-\sqrt{C})>0.
\end{equation}
Hence, all critical points of $f$ in the interior of $M$ are
non-degenerate local minima on $M$, and if $p\in \mbox{Int}(M)$
is a local minimum of $f$, then $M$ lies outside $\B (|p|)$
locally around $p$, touching $\esf ^2(|p|)$ only at $p$. Suppose
$f$ admits an interior critical point $p$. Since $M$ is connected,
we can choose a regular value $R_1>1$ of $f$ large enough so that $p$ lies in the same
component $M_1$ of $M\cap \overline{\B }(R_1)$ as $\partial M$. Since $p$
is a non-degenerate local minimum and $f$ has only non-degenerate critical
points, then $f$ is a Morse function on $M$ and Morse theory implies that $f|_{M_1}$
must have an index-one critical point, which is impossible. Therefore,
$f$ has no local minima on $M$ except along $\partial M$ where it attains its
global minimum value. Hence, $M$ intersects every sphere
$\esf^2(r)$, $r\geq 1$, transversely in a connected simple closed
curve, which implies that $M$ is an annulus.

If $M$ has finite total curvature, then it must be asymptotic to an
end of a plane or of a catenoid, thus either the lemma is proved or
$M$ has infinite total curvature.

A general technique which we will use in later sections to obtain compactness of
sequences of minimal surfaces is the following (see
e.g., Meeks and Rosenberg~\cite{mr8}): If $\{ M_n\} _n$ is a sequence of minimal surfaces
properly embedded in an open set  $B\subset \R^3$, with their
curvature functions $K_{M_n}$ uniformly bounded on compact subsets of $B$,
then a subsequence converges uniformly on compact subsets of $B$ to a minimal
lamination of $B$ with leaves that have the same bound on the
curvature as the surfaces $M_n$.

Suppose that the lemma fails. In this case, there exists a sequence
of positive numbers $C_n\to 0$ and minimal annuli $M_n$ satisfying
the conditions of the lemma, such that $M_n$ has infinite total
curvature and $|K_{M_n}|R^2 \leq C_n$. Since the $M_n$ are annuli
with infinite total curvature, the Gauss-Bonnet formula implies that
there exists a sequence of numbers $R_n\to \infty $ such that the
total geodesic curvature of the outer boundary of  $M_n\cap \B
(R_n)$ is greater than $n$. After extracting a subsequence, the
surfaces $\widetilde{M}_n=\frac{1}{R_n}M_n$ converge to a minimal lamination
${\cal L}$ of $\R^3-\{ \vec{0}\} $ that extends across $\vec{0}$ to
a lamination of $\R^3$ by parallel planes (since $|K_{\widetilde{M}_n}|R^2
\leq C_n$ and $C_n\to 0$ as $n\to \infty $).
Furthermore, ${\cal L}$ contains a plane $\Pi $ passing through
$\vec{0}$. Consider the great circle $\G =\Pi \cap \esf^2(1)$ and
let $\G (\ve )$ be the $\ve $-neighborhood of $\G $ in $\esf ^2(1)$,
for a small number $\ve >0$. Each $\widetilde{M}_n$ transversely
intersects $\esf^2(1)$ in a simple closed curve $\a _n$ and the
Gauss map of $\widetilde{M}_n$ along $\a _n$ is almost constant and
parallel to the unit normal vector to $\Pi $. Clearly, for $n$
sufficiently large, either $\widetilde{M}_n\cap \G (\ve )$ contains
long spiraling curves that join points in the two components of
$\partial \G (\ve )$ or it consists of a single closed curve which
is $C^2$-close to $\G $. This last case contradicts the assumption
that the total geodesic curvature of $M_n\cap \esf^2(R_n)$ is
unbounded. Hence, we must have spiraling curves in $\widetilde{M}_n
\cap \G (\ve )$. In this case, there are planes $\Pi _+,\Pi _-$ in
${\cal L}$, parallel to $\Pi $, such that $\partial \G (\ve )=(\Pi
_+\cup \Pi _-)\cap \esf^2(1)$. In a small neighborhood $U$ of $(\Pi
_+\cup \Pi _-)\cap \B (2)$ which is disjoint from $\Pi $, the
surfaces $\widetilde{M}_n\cap U$ converge smoothly to ${\cal L}\cap
U$. Since $(\Pi _+\cup \Pi _-)\cap \B (2)$ is simply connected, then
a standard monodromy lifting argument implies that for $n$ large,
$\widetilde{M}_n \cap
\overline{\B }(1)$ contains two compact disks in $U$ which are close
to $(\Pi _+\cup \Pi _-)\cap \overline{\B }(1)$. This contradicts the
fact that each $\widetilde{M}_n$ intersects $\esf^2(1)$ transversely in just one
simple closed  curve (see the first paragraph of this proof). This
contradiction completes the proof of the lemma.
\end{proof}

\begin{definition} \label{def4}
{\rm
We denote by $K_{\cal L}\colon{\cal L}\to \R $ the
Gaussian curvature function of a lamination ${\cal L}$.
A lamination ${\cal L}$ of $\R^3-\{ \vec{0}\} $
is said to have {\it quadratic decay of curvature} if $|K_{\cal
L}|R^2\leq C$ on ${\cal L}$ for some $C>0$.
}
\end{definition}

Our next goal is to show properness in $\R^3-\{ \vec{0}\} $ for every leaf of
a non-flat minimal lamination ${\cal L}$ of $\R^3-\{ \vec{0}\} $ with quadratic
decay of curvature (Proposition~\ref{propos1} below). The proof of this property
is a delicate technical argument, which we will break into separate statements.

\begin{lemma}
  \label{lema2new}
Let ${\cal L}$ be a non-flat minimal lamination of
$\R^3-\{ \vec{0}\} $ with quadratic decay of curvature. Suppose that $L$ is a leaf
of ${\cal L}$ which is not proper in $\R^3-\{ \vec{0}\} $. Then after a
rotation in $\R^3$, $L$ is contained in $H^+= \{ (x_1,x_2,x_3)\in \R^3\ | \ x_3>0\} $, and
\[
\mbox{\rm lim}(L)=\{ \mbox{limit points of $L$ in $\R^3-\{ \vec{0}\} $}\} =\{ x_3=0\} -\{ \vec{0}\} .
\]
In particular, $L$ is proper in $H^+$.
\end{lemma}
\begin{proof}
As $L$ is not proper in $\R^3-\{ \vec{0}\} $, then $L$ is not flat and
there exists $p\in \lim(L)\subset \R^3-\{ \vec{0}\} $.
Let $L'$ be the leaf of ${\cal L}$  that contains $p$. Since $L'\cap
\lim(L)$ is closed and open\footnote{\label{footnote}Openness follows since a neighborhood
of every point of $L'\cap \lim(L)$ can be written as the limit of a sequence
of graphs, all of them contained in $L$. Note that the same argument
shows that if $p$ is a limit point of a leaf $L$ of a lamination, then the leaf passing
through $p$ consists entirely of limit points of $L$.}
in $L'$, then $L'\subset \lim (L)$. In
particular, by Corollary~\ref{sc1}, $L'$ is either a plane or a
plane punctured at the origin, and $L$ is contained in one of the
half-spaces determined by $L'$. If $L'$ does not pass through
$\vec{0}$, then there exists $\ve >0$ such that the $\ve $-neighborhood $L'(\ve )$
of $L'$ is at positive distance from $\vec{0}$. Since $|K_L|R^2\leq C$ for some
$C>0$, then $L\cap L'(\ve )$ has bounded curvature, which is
impossible by the statement and proof of Lemma 1.3 in \cite{mr8};
for the sake of completeness we now sketch this argument. Taking $\ve
$ small, each component $\Omega $ of $L\cap L'(\ve )$ is a
multigraph over its orthogonal projection to $L'$. Actually $\Omega
$ is a graph over its projection on $L'$ by a separation argument.
Thus, $\Omega$ is proper in $L'(\ve )$, and the proof of the
Half-space Theorem~\cite{hm10} gives a contradiction. Hence, the
plane $\overline{L'}$ passes through $\vec{0}$. This argument also
shows that $L'$ equals $\lim (L)$ (otherwise we obtain a second
punctured plane $L''\subset \lim (L)$ which is a leaf of ${\mathcal L}$
and which passes through $\vec{0}$, which contradicts that $L'\cap
L''$ is empty). Now the lemma is proved.
\end{proof}

Given $\de >0$, let $C_{\de }=\{ (x_1,x_2,x_3)\ | \ x_3^2=\de ^2(x_1^2+x_2^2)\} \cap H^+$
(positive half-cone) and $C_{\de }^-$ the region of $H^+$ below $C_{\de }$.
\begin{lemma}
  \label{lema1new}
Let $L\subset H^+$ be a connected minimal surface which is complete
outside the origin and whose Gaussian curvature $K_L$ satisfies $|K_L|R^2\leq C$
for some $C>0$. Then for any $\ve >0$ small, there exists a $\de >0$ such
that in $L\cap C_{\de }^-$, the inequality $|\nabla _Lx_3|\leq \ve
\frac{x_3}{R}$ holds.
\end{lemma}
\begin{proof}
   A consequence of $|K_L|R^2\leq
C$ is that for all $\ve >0$, there exists $\de >0$ such that for all
$p\in L\cap C_{\de }^-$, then the angle that the tangent space to
$L$ at $p$ makes with the horizontal is less than $\ve $.

Consider the conformal change of metric
$\widetilde{g}=\frac{1}{R^2}g$ that we used in the proof of Lemma~\ref{lema1},
where $g$ is the induced metric on $L$ by the usual inner product of $\R^3$.
Recall that $\widetilde{g}$ was proven to be complete on $L$.
Then, we can apply Theorem 6 in Cheng-Yau~\cite{chya1} to the harmonic function $x_3$
on $(L,\widetilde{g})$  to obtain
\begin{equation}
\label{CY1}
\widetilde{g}(\widetilde{\nabla }x_3,\widetilde{\nabla }x_3)(x)^{1/2}\leq
\alpha \, x_3(x)\left( |\widetilde{K}_{\infty }|+\frac{1}{a}\right) ,
\end{equation}
where $\widetilde{\nabla }x_3$ is the gradient of $x_3$ with respect to $\widetilde{g}$,
$x\in L\cap C_{\de }^-$, $\a >0$ is a universal constant and
$\widetilde{K}_{\infty }\in \R $ is a lower bound for the Gaussian curvature of
the geodesic disc $\widetilde{D}_L(x,a)$ in $(L,\widetilde{g})$
centered at $x$ with radius $a>0$.
As $\widetilde{g}$ is complete on $L$, we can take $a\geq \frac{2\a }{\ve }$.
For this value of $a$, we can choose $\de >0$ small enough so that
$|\widetilde{K}_{\infty }|<\frac{\ve}{2\a }$ in $\widetilde{D}_L(x,a)$ (this follows
since for fixed $a>0$ and $\de >0$ arbitrarily small, the geodesic disks $\widetilde{D}_L(x,a)$ centered
at $x\in L\cap C^-_{\delta }$ can be considered to be domains in $(L,g)$ which are locally
approximated arbitrarily well in the $C^2$-norm as small $x_3$-graphs over
a domain in the plane $\{ x_3=0\} -\{ \vec{0}\} $,
and the restriction of $\widetilde{g}$ to $\{ x_3=0\} -\{ \vec{0}\} $ is a flat metric).
As $\widetilde{\nabla }x_3=R^2\nabla x_3$, then~(\ref{CY1}) transforms into
$R|\nabla x_3|\leq
\alpha \, x_3(x)\left( |\widetilde{K}_{\infty }|+\frac{1}{a}\right) <\ve \, x_3(x)$,
which proves the lemma.
\end{proof}

Next we analyze the geometry of each component of $L\cap C_{\de }$ for
$\de >0$ sufficiently small, where $L$ satisfies the hypotheses and
conclusions of Lemma~\ref{lema2new}. Note that Lemma~\ref{lema1new} applies in this
setting and guarantees that for $\ve >0$ fixed and small, and $\de >0$
sufficiently small, $L$ intersects $C_{\de }$ transversely in a
small angle that is uniformly bounded away from zero and each
component of $L\cap C_{\de }$ is locally a radial graph over the
circle $C_{\de }\cap \{ x_3=1\} $. Furthermore, in the natural
polar coordinates in $C_{\de }$, the radial lines intersect the
collection of curves $L\cap C_{\de }$ almost orthogonally. Let
$\Lambda $ be the set of components of $L\cap C_{\de }$. Then
any $\G \in \Lambda $ is of one of the following two types, see
Figure~\ref{figure1}:
\begin{description}
\item[{\sc Type I.}] $\G $ is a closed almost horizontal curve.
In this case, any other $\G '\in \Lambda $ is also of type I, and
there are an infinite number of these curves, converging to $\{
\vec{0}\} $.
\item[{\sc Type II.}] $\G $ is a spiraling curve (with almost
horizontal tangent vector) limiting down to~$\{ \vec{0}\} $. $\G $
rotates infinitely many times around $C_{\de }$, and any other $\G
'\in \Lambda $ is of type~II. Note that in this case,
$\Lambda=\{\G_1, \ldots, \G_k\}$ has a finite number of these
spiraling components.
\end{description}
\begin{figure}
\begin{center}
\includegraphics[width=11.9cm,height=4cm]{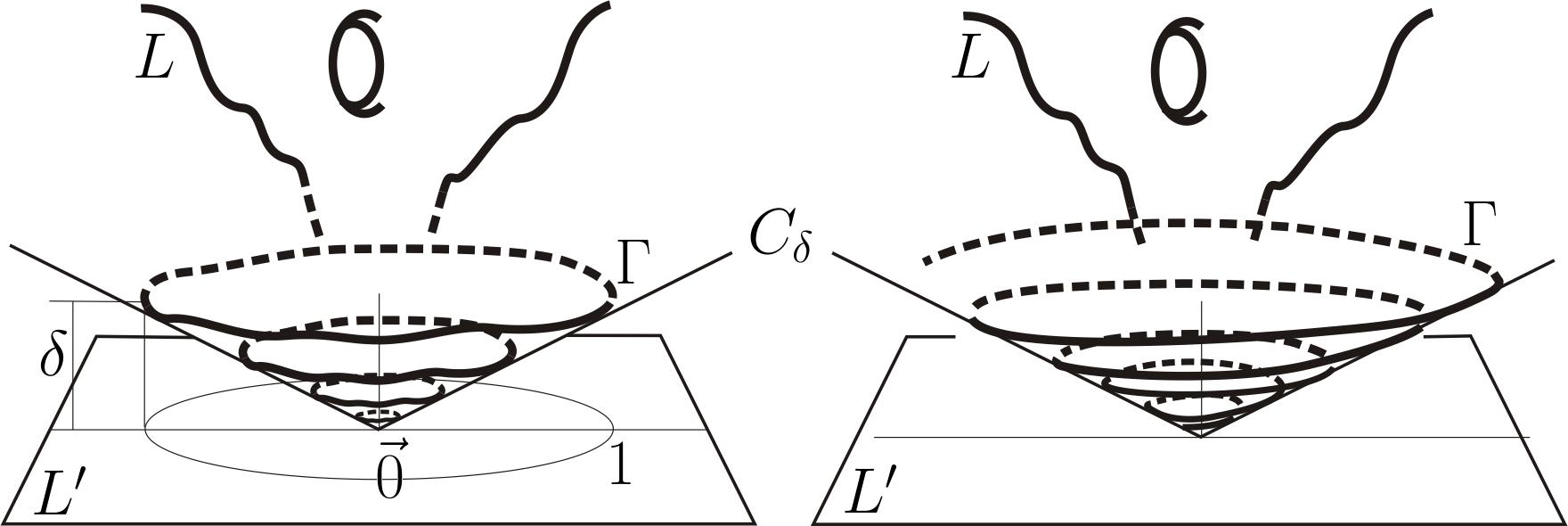}
\caption{Curves $\G \in \L $ of type I (left) and one curve of type II (right).}
\label{figure1}
\end{center}
\end{figure}
\par
\noindent
Observe that $L\cap C_{\de '}$ has the same pattern as $\Lambda $,
for each $\de '\in (0,\de )$. {\bf Our next goal is to find a contradiction in each of the
two cases listed above.} These contradictions will lead to the conclusion that a
leaf component $L$ satisfying the hypotheses and conclusions of Lemma~\ref{lema2new}
cannot exist. This will be the content of Proposition~\ref{propos1} at the end of this section.

\subsection{Suppose the curves in $\Lambda $ are of type I.}
\label{subsect1}

Let $\G \in \Lambda $. Denote by $E(\G )$ the component
  of $L\cap C_{\de }^-$ whose boundary contains $\G $. Lemma~\ref{lema1new}
implies that the third component of the unit normal vector to $E(\G )$ has a fixed
non-zero sign on $E(\G)$, and so, $E(\G)$ is locally a graph over its vertical
projection to $\{ x_3=0\} $. Since $E(\G )$ separates
$C_{\de }^-$ (because $E(\G )$ is properly embedded in the closure of
$C_{\de }^-$ and $\G \subset E(\G )$ generates the first homology group of $C_{\de }^-(1)$),
then $E(\G )$ is a global graph over its projection to $\{ x_3=0\} $.

\begin{assertion}
\label{ass4.5}
For $\de >0$ sufficiently small, the boundary of $E(\G )$ equals $\G $ for each
$\G \in \Lambda $.
\end{assertion}
\begin{proof}
Observe that otherwise there exists a point
$p=p(\de )\in E(\G )\cap C_\de $ where the tangent plane $T_pL$ is steeper than
the tangent plane to $C_{\de }$, i.e. $\de <|\nabla _Lx_3|$. Using again Lemma~\ref{lema1new},
\[
|\nabla _Lx_3|(p)\leq \ve \frac{x_3(p)}{R(p)}=\ve \frac{\de }{\sqrt{1+\de ^2}},
\]
where we have used that $p\in C_{\de }$ in the last equality. Therefore,
$\de <\frac{\ve \de }{\sqrt{1+\de ^2}}$, or equivalently,
$\sqrt{1+\de^2}<\ve $, which is impossible for $\ve >0$ small enough.
This proves the assertion.
\end{proof}

\begin{assertion}
\label{ass4.6}
There exists a sequence of points
  $\{ q_n\} _n\subset L$ converging to~$\vec{0}$ such that
for all $n\in \N$, we have $(|K_L|R^2)(q_n)\geq 1$.
  \end{assertion}
\begin{proof}
Reasoning by contradiction, suppose that there exists $r>0$
   small such that $|K_L|R^2<1$ in $L\cap \overline{\B }(r)$.
   By the arguments in the proof of Lemma~\ref{lemma2}, $f=R^2$
   is a Morse function with only local minima in $L\cap \overline{\B }(r)$,
and so,
$L\cap \overline{\B }(r)$ consists of a non-empty family of compact disks
and non-compact annuli with boundary on
$\esf^2(r)$ and which are proper in $\overline{\B }(r)-\{ \vec{0}\} $.
  Let $\Omega $ be one of these components and
suppose $\Omega $ is an annulus. If
$\Omega $ is conformally $\D ^*$, then $\Omega $ extends smoothly
  across $\vec{0}$, which contradicts the maximum principle since $L$ is
contained in $\{ x_3>0\} $. If $\Omega $
   is conformally $\{ \ve <|z|\leq 1\} $ for some $\ve >0$, then each
   coordinate function of $\Omega $ can be reflected in $\{ |z|=\ve \} $
   by Schwarz's reflection principle, defining a branched conformal
   harmonic map on
a larger annulus that maps the entire curve $\{ |z|=\ve \} $ to a single
  point, which is impossible.
This means that every component in $L\cap \overline{\B }(r)$ is a compact
disk. As the points of $\esf ^1(r)\times \{ 0\} $ are limit points of $L$,
we conclude that there exists
a sequence of boundary curves $\g _n$ of these disks components of $L\cap
\overline{\B }(r)$ that converges to  $\esf^1(r)\times \{ 0\} $, and such that
for $n$ large, $\g _n$ is the boundary of an exterior non-compact minimal
  graph over its projection to $\{ x_3=0\} $. This clearly contradicts that
$L$ is connected and finishes the proof of the assertion.
\end{proof}

Since $|K|R^2\leq C$ on the minimal  laminations
$\frac{1}{|q_n|}[L\cup (\{ x_3=0\} -\{ \vec{0}\} ]$,
a subsequence of these laminations converges to a minimal lamination
${\cal L}_1$ of $x_3^{-1}([0,\infty ))-\{ \vec {0}\} $ that contains $\{ x_3=0\} -\{ \vec{0}\}$.
By Assertion~\ref{ass4.6}, ${\cal L}_1$ also contains a non-flat leaf $L_1$ passing through a point
in $\esf^2(1)$. Since ${\cal L}_1$ is a lamination outside the origin, then
$L_1$ is complete outside the origin. Furthermore, $\vec{0}$ is in the closure of $L_1$ in $\R^3$
(otherwise the inequality $|K_{L_1}|R^2\leq C$ implies that
$L_1$ is a complete minimal surface with bounded Gaussian curvature,
which contradicts that $L_1$ is contained in a halfspace, see Xavier~\cite{xa5}).

\begin{assertion}
\label{ass4.7}
  In the above situation, $L_1$ is proper in $H^+$ and $\lim (L_1)=\{ x_3=0\} -\{ \vec{0}\} $.
\end{assertion}
\begin{proof}
  By Lemma~\ref{lema2new} and since $L_1\subset H^+$, it suffices to prove that
$L_1$ is not proper in $\R^3-\{ \vec{0}\} $. Arguing by contradiction, suppose that
$L_1$ is proper in $\R^3-\{ \vec{0}\} $. Since $L_1$ is disjoint from the compact
set $\esf ^1(1)\times \{ 0\} $, then the distance from $\esf ^1(1) \times \{ 0\} $
to $L_1$ is greater than some $d>0$. In particular,
the cylinder $\esf^1(1)\times [0,d]$ is disjoint from
$L_1$. Fix $n\in \N$. Since the curves $\esf^1(1)\times \{ d \} $ and
$\esf^1(\frac{1}{n})\times \{ 0\} $ are homotopic in a component of
$\{ x_3\geq 0\} -L_1$, then using $L_1$
as a barrier we obtain a compact least-area annulus $A(n)$ disjoint from
$L_1$ with these curves as boundary. The annulus $A(n)$
is a catenoid but no such catenoid exists for $n$ large enough,
since $d >0$ is fixed. This contradiction proves the assertion.
\end{proof}
By Assertion~\ref{ass4.7}, Lemma~\ref{lema1new} and the arguments in the
paragraph just after its proof, the intersection $\Lambda _1$ of $L_1$ with $C_{\de }$
consists of curves of type I or II. If the curves in $\Lambda _1$ are of type~II,
then the corresponding spirals produce after shrinking back to $L$ spiraling curves of
type II on $\Lambda $, which is contrary to the hypothesis. Thus, $\Lambda _1$ consists
of curves of type I.

By our previous description of type I curves, the components in $\Lambda _1$,
close to the origin are closed, almost horizontal curves that are naturally ordered by
their distances to $\vec{0}$ and have $\{ \vec{0}\} $ as limit set.
Furthermore, Assertion~\ref{ass4.5} gives that for $\de >0$ sufficiently small,
every $\G _1\in \Lambda _1$ bounds an annular end $G_{\G _1}$ of $L_1\cap C_{\de}^-$
which is a graph over the exterior of the vertical projection of $\G _1$
in $\{ x_3=0\} $.

\begin{assertion}
\label{ass4.8}
There exists a compact horizontal disk $\Delta _1\subset H^+$ such that the following properties hold.
\begin{enumerate}
\item $\Delta _1\cap L_1$ consists of a (non-zero) finite number of
simple closed  curves in $\Delta _1-\partial \Delta _1$.
\item $\Delta _1$ can be extended to a topological plane $P(\Delta _1)$ which is a global
graph over $\{ x_3=0\} $ and with $P(\Delta _1)\cap L_1=\Delta _1\cap L_1$.
\end{enumerate}
\end{assertion}
\begin{proof}
Recall that given $\G _1\in \L _1$, the end $G_{\G_1}$ of $L_1\cap C_{\de}^-$
is annular and graphical, hence
it has finite total curvature and it is planar or catenoidal (with positive logarithmic growth
because $L_1\subset H^+$).

If there exists a curve $\G _1\in\Lambda _1$ such that its corresponding
graphical annular end $G_{\G _1}$ of $L_1\cap C_{\de}^-$ is planar, then any curve
$\G '_1\in \Lambda _1-\{ \G_1\} $ at smaller distance from $\vec{0}$ than
$\G _1$ also bounds a planar graphical annular end $G_{\G '_1}$
of $L_1\cap C_{\de}^-$. In this case, since between consecutive
planar ends of $L_1$ we can always find a horizontal plane
$P$ that intersects $L_1$ transversally in a compact set,
our claim holds by taking an appropriate disk $\Delta _1$ in
$P$ and by letting $P(\Delta _1)=P$.

Suppose now that a curve $\G _1\in \Lambda _1$ bounds a
catenoidal end $G_{\G _1}\subset L_1$. Take a
horizontal plane $P\subset H^+$ whose height is large enough so that
$G_{\G _1}$ intersects $P$ transversely at an almost-circle $\G _1'$.
Consider the closed horizontal disk $\widehat{\Delta }_1\subset P$
bounded by  $\G _1'$. Now define ${\Delta }_1\subset P$ to be
a slightly smaller closed disk in $P$ which is contained in the interior of
$\widehat{\Delta }_1$ and such that
\[
W_1:=\Delta_1 \cap L_1 =(\widehat{\Delta}_1\cap L_1) -\G _1'
\]
is contained in the interior of $\Delta _1$ (we can assume $W_1\neq \mbox{\O}$
  by taking $\G_1$ close enough to $\vec{0}$).
In this case, the topological plane $P(\Delta _1)$ can be taken
to be the union of $\Delta _1$
with a minimal annular graph lying strictly above $G_{\G _1}$
and close to it. Observe
that the set of curves $W_1$ is the boundary of a proper, possibly disconnected
subdomain $L_1 (W_1)$ of $L_1$ which lies above $P$; see Figure~\ref{figure3}.
\begin{figure}
\begin{center}
\includegraphics[width=13cm]{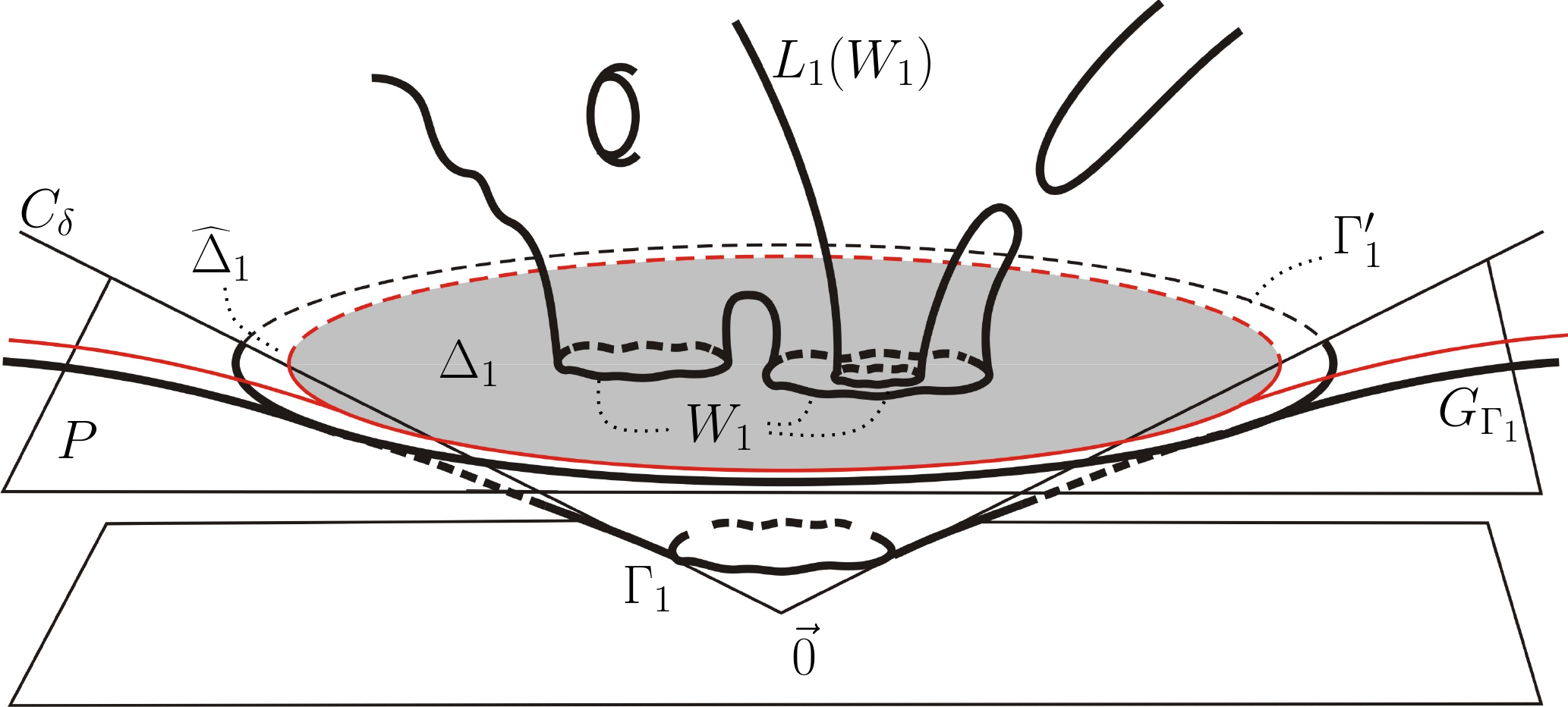}
\caption{The shaded area is the compact horizontal disk $\Delta _1$,
  in the case that the end $G_{\G _1}$ is catenoidal.}
\label{figure3}
\end{center}
\end{figure}
This completes the proof of the assertion.
\end{proof}

To find the desired contradiction in this case of type I curves,
we will use a flux argument.  Recall that for a
curve $\g \subset L$, the {\it flux of $x_3$} along $\g $ is defined as
\[
\mbox{Flux}(\nabla x_3,\g )=\int _{\g }\frac{\partial x_3}{\partial \eta },
\]
where $\eta $ denotes a unit conormal to $L$ along $\g $.

In the next flux argument, we will just consider the case where the annular graphical ends
$G_{\G _1}\subset L_1\cap C_{\de}^-$ for $\G_1\in \Lambda _1$ are catenoidal
(the planar end case is similar and we leave the details to the reader).
For $n$ large, the shrunk disks $D(n)=|q_n|{\Delta }_1$ intersect $L$
transversally in a set $W(n)$ which consists of a finite number of closed curves.
Furthermore, $W(n)$ bounds a proper, possibly disconnected subdomain
$L(W(n))$ of $L$ and $L(W(n))$ lies above the horizontal plane that contains $D(n)$.
By construction, $L(W(n))$ is the portion of $L$ above a topological plane
$P(n)$ which is a global graph over $\{ x_3=0\} $ with $P(n)\cap L=W(n)$.
Note that $L(W(n))\subset L(W(n+m))$ for every $n,m\in \N$ with $n$ large.
Since $x_3$ is proper on $L(W(n+m))$,
the absolute value of the flux of $\nabla x_3$ across $\partial L(W(n+m))$
is not less than the absolute value of flux of $\nabla x_3$ across $\partial L(W(n))$,
which is positive. This is a contradiction, since the length of
$\partial L(W(n+m))$ converges to $0$ as $m\to \infty $. This contradiction finishes
the analysis when the curves in $\Lambda $ are of type I.

\subsection{Suppose the curves in $\Lambda $ are of type II.}
\label{subsect2}

Take a component $\G $ in $\Lambda $. By embeddedness, all of the curves in
$\Lambda -\{ \G \} $
have disjoint arcs trapped between one complete turn of $\G $.
Since $L$ is proper in $H^+$, we have
the number of curves in $\Lambda $ is finite, say
$\Lambda =\{ \G (1),\ldots ,\G (k)\} $.

Consider the horizontal plane at height $\de $, which intersects $C_{\de }$
in a circle of radius $1$. After a small perturbation of $\de $, we may assume that
$\{ x_3=\de \} $ intersects $L$ transversely and intersects each $\G (i)$
transversely, $i=1,\ldots ,k$. Using natural cylindrical coordinates $(r,\theta ,x_3)$
in $\{ x_3\geq 0\} -\{ x_3-\mbox{axis}\} $, each spiraling curve
$\G (i)$ can be parameterized by the angle as $(r_i(\t ),\t ,\de \, r_i(\t ))$ for
$\t \in (-\infty ,\infty )$, where $r_i\colon \R\to (0,\infty )$
is a smooth function satisfying
\[
\lim _{\theta \to -\infty }r_i(\theta )=\infty ,\qquad
\lim _{\theta \to +\infty }r_i(\theta )=0.
\]
Furthermore, we may assume that the largest values of $\t $
such that the spiraling curves $\G (i)$ intersect $\{ x_3=\de \} $ are given by
$0\leq \t _1<\t _2<\ldots <\t_k<2\pi $. This means that
$r_i(\t _i)=1$, $r_i(\t )<1$ for $\t >\t _i$, $i=1,\ldots ,k$.

As in the case of type I curves, we will use a flux argument to find a contradiction.
Given $t>0$, we define the function
\begin{equation}
\label{eq:fluxt}
F(t)=\sum _{i=1}^k\mbox{Flux}\left( \nabla x_3,\G (i)|_{[\t _i,\t _i+t]}\right) \in \R ,
\end{equation}
where the unit conormal vector with respect to which the above flux is computed is
pointing inwards $L\cap C_{\de }^-$.

\begin{lemma}
\label{lema4.9}
  The function $F\colon (0,\infty )\to \R $ is bounded.
\end{lemma}
\begin{proof}
  Arguing by contradiction, we may assume that there exists a sequence $t_n\to \infty $ such that
$|F(t_n)|\to \infty $. Consider the sequence $\l _n=|\G (1)(\t _1+t_n)|>0$, which converges to zero
as $n\to \infty $. After passing to a subsequence, the laminations $\frac{1}{\l_n}[L\cup (\{ x_3=0\} -\{ \vec{0}\} )]$
converge to a minimal lamination ${\cal L}_1$ of $\{ x_3\geq 0\} -\{ \vec{0}\} $ with quadratic decay of curvature.
Note that ${\cal L}_1$ contains a leaf $L_1$ which passes through the limit point $q_{\infty }$ of the sequence
$\frac{1}{\l _n}\G (1)(\t _1+t_n)$, which lies in $C_{\de }\cap \esf^2(1)$
(in particular, $x_3(q_{\infty })=\frac{\de }{\sqrt{1+\de ^2}}>0$).

We claim that ${\cal L}_1$ is not flat. Otherwise, ${\cal L}_1$ contains the plane
$\{ x_3=\frac{\de }{\sqrt{1+\de ^2}}\} $ as a leaf. Therefore, for $n$ large
$\frac{1}{\l_n}L$ contains an almost horizontal compact
disk arbitrarily close to the disk $\{ (x_1,x_2,\frac{\de }{\sqrt{1+\de ^2}} )\ | \ x_1^2+x_2^2\leq
1+\frac{1}{1+\de ^2}\}$. Since $L$
is embedded, this contradicts the existence of the proper spiraling curves $\G (i)$, $i=1,\ldots ,k$. Therefore,
${\cal L}_1$ is not flat. The same argument proves that the leaf $L_1$ of ${\cal L}_1$ passing through $q_{\infty }$
is not flat.

Since $L_1$ is not flat, then $L_1$ contains $\vec{0}$ in its closure (argue as in the last sentence before the
statement of Assertion~\ref{ass4.7}). The proof of Assertion~\ref{ass4.7} applies without changes here and gives
that lim$(L_1)=\{ x_3=0\} -\{ \vec{0}\} $ and $L_1$ is proper in $H^+$, and thus Lemma~\ref{lema1new} insures that
the intersection $\Delta _1$ of $L_1$ with $C_{\de }$ consists of curves of type I or II. Type $I$ curves
of $\Delta _1$ cannot occur, since we have explained that each of the types I and II persists
after changing scale and taking limits.

The same analysis in the last paragraph shows that every leaf of ${\cal L}_1$ different from $\{ x_3=0\} -\{ \vec{0}
\} $ contains $\vec{0}$ in its closure, has $\{ x_3=0\} -\{ \vec{0}\} $ as limit set in $\R^3-\{ \vec{0}\} $, is proper in $H^+$
and intersects $C_{\de }$ in finitely many curves of type II.
The non-existence of limit leaves of ${\cal L}_1$ in $H^+$
(by Corollary~\ref{sc1}) implies that ${\cal L}_1$ consists of a finite number of leaves.
The same argument (by Corollary~\ref{sc2}) shows that
the multiplicity of the limit $\frac{1}{\l_n}[L\cup (\{ x_3=0\} -\{ \vec{0}\} )]\to {\cal L}_1$ is one; consequently,
$\mathcal{L}_1\cap C_{\de }$ consists of precisely $k$ spiraling curves,
each of which is a limit of a rescaled spiraling curve in $L$.
This correspondence allows us to label the curves in $\mathcal{L}_1\cap C_{\de }$
with the same indices as the related curves in $L\cap C_{\de }$.

Consider an auxiliary smooth
compact disk $D\subset H^+-C_{\de }^-$ with $\partial D\subset C_{\de }$, satisfying the following
properties:
\begin{enumerate}[(P1)]
\item $D$ intersects transversely the union of the leaves in ${\cal L}_1$.
\item $\partial D$ intersects transversely each of the
spirals in $L_1\cap C_{\de }$ in a single point.
\end{enumerate}
 After shrinking back to the original scale, we find a sequence of compact disks $\l _nD\subset H^+-C_{\de }^-$
 which for $n$ sufficiently large, satisfy the properties
\begin{enumerate}[(P1)']
\item $\l_nD$ intersects transversely $L$.
\item $\l _n\partial D$  intersects transversely each of the
spirals $\G (1),\ldots ,\G (k)\subset L_1$ in points of the form $\G (i)(\t _i+t_{i,n})$.
We can assume without loss of generality that $q_{\infty }$ is the limit of rescaled points
in $\G (1)(\t _1+t_{1,n})$.
\item $|t_n-t_{1,n}|\to 0$ as $n\to \infty$ (since the angle coordinate is invariant under change of scale and
the points $\G (1)(\t _1+t_n)$ converge after rescaling by $1/\l _n$ to $q_{\infty }$ as $n\to \infty $).
\item $|t_n-t_{i,n}|$ is bounded independently of $i,n$ for each $i=2,\ldots ,k$ (this holds since $|\t _i-\t _1|\leq 2\pi $
and by property (P3)').
\end{enumerate}
As $|F(t_n)|\to \infty $ by hypothesis and the sum of lengths of the $\G (i)$ from
$\t _i+t_n$ to $\t _i+t_{i,n}$ tends to zero as $n\to \infty $ (by properties (P3)', (P4)'), then
\begin{equation}
\label{fluxtin}
\sum _{i=1}^k\mbox{Flux}\left( \nabla x_3,\G (i)|_{[\t _i,\t _i+t_{i,n}]}\right) \to \infty \mbox{ as }n\to \infty .
\end{equation}
We will next obtain the desired contradiction as an application of the fact that the total flux of a
compact minimal surface along its boundary vanishes. To do this, consider the compact minimal surface
with boundary $\wh{L}(n)\subset L$ bounded by $\l _1D$,
$\l _nD$ and $C_{\de }$, where $n$ is chosen large enough so that $(\l _1D)\cap (\l _nD)=\mbox{\O }$. Since the flux of
$\nabla x_3$ along $\l _1D$ is a fixed number and the flux of $\nabla x_3$ along $\l _nD$ tends to zero as $n\to \infty $
(because the length of $(\l _nD)\cap L$ tends to zero), then we conclude that the last displayed summation is bounded as
$n\to \infty $, which is a contradiction.
\end{proof}

By the description above, each of the components of $L\cap C_{\de }^-$ can be considered to be a graph of a
function $u_i=u_i(r,\t )$ defined over a region $\Omega _i$
in the universal cover of the punctured plane $\{x_3=0\} -\{ \vec{0}\} $, of the type
\[
\Omega _i=\{ (r,\t )\ | \ r>r_i(\t ),\ \t \in \R \} ,
\]
where the function $r_i(\theta )$ was defined before the statement of Lemma~\ref{lema4.9}.
In particular, the restriction of $u_i$ to the quadrant
$\{ (r,\t )\ | \ r\geq 1,\ \t \geq \t _1\} $ is an
$\infty $-valued graph in the sense of Colding-Minicozzi.

Consider the compact region $R=\overline{C_{\de }^-}\cap \{ (x_1,x_2,x_3)\ | \
x_1^2+x_2^2\leq 1\} $. By Lemma~\ref{lema1new}, each component of the intersection of
$L\cap R$ with any vertical halfplane $\{ \t =\mbox{constant}\} $ has length not greater than 2 for $\de >0$
sufficiently small. This property together with Lemma~\ref{lema4.9} and the Divergence Theorem implies that
the following function is bounded:
\begin{equation}
\label{eq:hatflux}
\wh{F}\colon (0,\infty )\to \R, \quad
\wh{F}(t)=\sum _{i=1}^k\mbox{Flux}\left( \nabla x_3,\g _{i,t}\right) ,
\end{equation}
where $\g_{i,t}$ is given in cylindrical coordinates by $\g _{i,t}(\theta )=(1,\t ,u_i(1,\t ))$, $\t \in [\t _i,\t _i+t]$.

The desired contradiction in this case of type II curves will come from application of a slight adaptation of
inequality (5.3) in the proof
of Theorem~1.1 in Colding and Minicozzi~\cite{cm26}. To be more precise, inequality (5.3) states
(with the notation in~\cite{cm26}):
\[
\sum _{j=1,2}\int _1^{r}\int _0^t\frac{|w_j(e^s,0)|}{t}ds\, dt-\sum _{j=1,2}\int _0^{2\pi }\int _{\partial \wt{D}_1\cap Q}
u_j(e^{x+i(y+\tau )}d\tau
-6\pi \sum _{j=1,2}u_j^3(1,0)\log r
\]
\[
\leq -\int _0^{2\pi }\int _1^r\left( \sum _{j=1,2}\int _{(\Sigma _j)_{1,1}^{\tau ,t+\tau }}
\frac{\partial x_3}{\partial \eta }\right) \frac{1}{t}dt\, d\tau \leq 2C_1\log r. \qquad (\star )
\]
In our application, we have $k$ $\infty $-valued minimal graphs instead of two; the last inequality in
$(\star )$ holds in our case by Lemma~\ref{lema4.9}, since the summation in the parenthesis is essentially the function
  $\wh{F}(t)$ defined in (\ref{eq:hatflux}). This is a contradiction (for $r$ large) since the
left-hand-side of $(\star)$ has three terms, of which the third one grows as $\log r$, the second one
is a fixed constant not depending on $r$ and the first one can be bounded below by a positive constant times
$(\log r)^2$ by integrating the first inequality in Theorem~1.3 of~\cite{cm26} (which is valid
in our situation since the $\infty $-valued minimal graphs are defined for $r\in [1,\infty )$ and
their gradient can be made arbitrarily small by choosing $\de >0$ small enough). This contradiction
finishes the analysis of type II curves.

\begin{remark}
{\rm In the above analysis of type I and II curves in $\Lambda $ we have used a flux argument based
on Colding-Minicozzi theory. In the appendix we provide a self-contained argument to get the
same conclusions as above, by showing that a non-proper leaf $L$ as in Lemma~\ref{lema2new} is
recurrent for Brownian motion,
which by the Liouville theorem applied to the positive third coordinate function $x_3|_L$, gives the desired contradiction
to the existence of such an $L$. See Proposition~\ref{propos1new} for details.
}
\end{remark}

We are now ready to prove the main result of this section, which
will be improved in Corollary~\ref{corollamin}.

\begin{proposition}
\label{propos1}
Let ${\cal L}$ be a non-flat minimal lamination of
$\R^3-\{ \vec{0}\} $ with quadratic decay of curvature. Then, any
leaf of ${\cal L}$ is a properly embedded minimal surface in
$\R^3-\{ \vec{0}\} $, and ${\cal L}$ does not contain flat leaves.
\end{proposition}
\begin{proof}
Arguing by contradiction, suppose $L$ is a leaf of ${\cal L}$ which
is not proper in $\R^3-\{ \vec{0}\} $. By Lemmas~\ref{lema2new} and~\ref{lema1new},
after a rotation we can assume $L\subset H^+$, lim$(L)=\{ x_3=0\} -\{ \vec{0}\} $
and for $\de >0$ sufficiently small, we have the description of the set $\Lambda $
of components of $L\cap C_{\de }$ given just after Lemma~\ref{lema1new}
in terms of curves of types I and II. By the above analysis, neither of these two cases
can occur. This contradiction shows that every leaf $L$ of $\mathcal{L}$ is properly embedded in $\R^3-
\{ \vec{0}\} $.

To finish the proof of the proposition, it remains to show that none of the leaves of
${\cal L}$ are flat. Arguing again by contradiction,
suppose $\widetilde{L}\in {\cal L}$ is a flat leaf and let
$\widehat{L}\in {\cal L}$ be a non-flat leaf. If $\widehat{L}$ does
not limit to $\vec{0}$, then $\widehat{L}$ has bounded curvature,
and so, it cannot be contained in a halfspace~\cite{xa5}.
This contradicts the Half-space Theorem as $\wh{L}$ lies
at one side of the closure of $\wt{L}$, which is a plane.
Hence, $\vec{0}$ is a limit point of $\widehat{L}$.

Now consider the sublamination $\widetilde{\cal L}=\{
\widehat{L},\widetilde{L}\} $. Suppose that $\widetilde{L}$ has
$\vec{0}$ in its closure and we will obtain a contradiction. After a
rotation, assume that $\widetilde{L}$ is the $(x_1,x_2)$-plane and
the third coordinate of $\widehat{L}$ is positive.
In this setting, the proof of Assertion~\ref{ass4.7} applies to $L_1=\wh{L}$
and gives the desired contradiction.

So we may assume that $\widetilde{L}$ is a plane which does not pass
through $\vec{0}$. As $\wh{L}$ is properly embedded in $\R^3-\{ \vec{0}\} $,
the proof of the Half-space Theorem~\cite{hm10} gives that
the distance between $\widehat{L}$ and
$\widetilde{L}$ is positive. Consider the plane $\Pi $ parallel to
$\widetilde{L}$ at distance $0$ from $\widehat{L}$. Since
$\widehat{L}$ is not a plane, $\Pi $ must go through the origin, and
we finish as before. Now the proof is complete.
\end{proof}

\section{The proof of Theorem~\ref{tt2}.}
\label{sectionproofremov}
We will divide the proof in five different cases.
\par
\vspace{.2cm}
{\sc Case I:} {\bf Suppose that $N=\R^3$, $p=\vec{0}$
{\rm (hence $\overline{B}_N(p,r)= \overline{\B} (r)$)} and ${\cal L}$
consists of a single leaf $M$ which is properly
 embedded in $\overline{\B} (r)-\{ \vec{0}\} $.}
\par
\vspace{.2cm}
\noindent
 In this case it is known that the area of $M$ is finite and $M$
 satisfies the monotonicity formula, see for instance Harvey and
 Lawson~\cite{hl1}. For the sake of
completeness, we give a self-contained proof in our setting.

For $0 < R_1\leq R_2\leq r$, let $A_M(R_1)=\mbox{Area}(M\cap \B
(R_1))$, $l_M(R_1)=\mbox{Length}(M\cap \esf^2 (R_1))\in (0,\infty ]$
and $A_M(R_1,R_2)=\mbox{Area}(M\cap [\B (R_2)-\B (R_1)])\in
(0,\infty )$. The Divergence Theorem applied to the vector field
$p^T=p-\langle p,\vec{n}\rangle \vec{n}$ (here $\vec{n}$ is the
Gauss map of $M$) gives:
\begin{equation}
\label{Stokes}
2A_M(R_1,R_2)=\int _{M\cap [\B (R_2)-\B (R_1)]}\mbox{div}_M(p^t)=
\int _{\partial _{R_1}}\langle p,\nu \rangle +
\int _{\partial _{R_2}}\langle p,\nu \rangle  ,
\end{equation}
where $\partial _{R_i}=M\cap \esf^2(R_i)$, $i=1,2$,
and $\nu $ is the unit exterior conormal vector to
$M\cap [\B (R_2)-\B (R_1)]$ along its boundary. The first integral
in the right-hand-side is not positive, and
Schwarz inequality applied to the second one gives $2A_M(R_1,R_2)\leq
R_2\, l_M(R_2)$. Taking $R_1\to 0$ and relabeling
$R_2$ as $R$, we have
\begin{equation}
\label{eq:removsingD}
2A_M(R)\leq R\, l_M(R)\quad \mbox{for all }R\in [0,r].
\end{equation}
In particular, the total area of $M$ is finite. Next we observe
that the monotonicity formula holds in our setting (i.e. $R^{-2}A_M(R)$ is
not decreasing for $R\in [0,r]$). To see this, note that
\begin{equation}
\label{eq:removsingB}
R^3\frac{d}{dR}\left( \frac{A_M(R)}{R^2}\right)=R\ A_M'(R)-2A_M(R).
\end{equation}
The coarea formula applied to the radial distance function $R$ to $\vec{0}$ gives
\begin{equation}
\label{eq:removsingC}
A_M'(R)=\int _{\partial _R}\frac{ds}{|\nabla R|}\geq l_M(R)
\end{equation}
where $\nabla R$ is the intrinsic gradient of $R|_M$ and $ds$ is the
length element along $\partial _R$. Now (\ref{eq:removsingD}),
(\ref{eq:removsingB}) and (\ref{eq:removsingC}) imply the monotonicity formula.

As an important consequence of the finiteness of its area together
with the monotonicity formula, $M$ has limit tangent cones at the
origin under expansions. To prove that $M$ extends to a smooth
minimal surface in $\overline{\B }(R_1)$, we discuss two situations
separately. In the first one (paragraph I.1 below)
we will deduce that $M$ has finite
topology, in which case the removability theorem is known (see
\cite{cs1}, although we also provide a proof of the
removability of the singularity in this situation),
 and to conclude the proof in Case I, we will show that
 the second situation (paragraph I.2 below) cannot hold.
\vspace{.2cm}

{\sc I.1.} Suppose there exist constants $C_1<1$ and $r'\leq r$ such that
$|K_M|R^2\leq C_1$ in $M\cap \B (r')$. Using the
arguments in the proof of Lemma~\ref{lemma2}, we deduce
that $M\cap \B (r')$ consists of a finite number of
annuli with compact boundary, transverse to the spheres
centered at the origin such that each of these annuli
has  $\vec{0}$ in its closure, together with a finite
number of compact disks. Thus, for $r'$ sufficiently small we may assume
that there are no such disk components. Let $A$ be one
of the annuli in $M\cap \overline{\B }(r')$. If $A$ is
conformally $\{ \ve <|z|\leq 1\} $ for some $\ve >0$,
then each coordinate function of $A$ can be reflected
in $\{ |z|=\ve \} $ (Schwarz's reflection principle),
 defining a conformal branched harmonic map which
carries the entire curve $\{ |z|=\ve \} $ to a single
point, which is impossible. Thus, $A$ is conformally
$\{ 0<|z|\leq 1\}$ and so, its coordinate
functions extend smoothly across $\vec{0}$, defining a
possibly branched minimal surface $A_0$ that passes
through~$\vec{0}$. If $\vec{0}$ is a branch point of
$A_0$, then $A$ cannot be embedded in a punctured
neighborhood of $\vec{0}$, which is a contradiction.
Therefore, $A_0$ is a smooth embedded minimal surface
passing through~$\vec{0}$. Since $M$ is embedded, the
usual maximum principle for minimal surfaces implies
that there exists only one such surface $A_0$, and
the theorem holds in this case.
\vspace{.2cm}

{\sc I.2.} Now assume that there exists a sequence $\{
p_n\} _n\subset M$ converging to $\vec{0}$ such that
 $1\leq |K_M|R^2(p_n)$ for all $n$, and we will obtain
 a contradiction. The expanded surfaces
 $\widetilde{M}_n=\frac{1}{|p_n|}M\subset \R^3-
 \{ \vec{0}\} $ also satisfy $|K_{\widetilde{M}_n}|R^2
 \leq C$. After choosing a subsequence, the
 $\widetilde{M}_n$ converge to a minimal lamination
 ${\cal L}_1$ of $\R^3-\{ \vec{0}\} $ with
 $|K_{{\cal L}_1}|R^2\leq C$. Furthermore,
 ${\cal L}_1$ contains a non-flat leaf $L_1$ passing
through a point in $\esf^2(1)$, where it has absolute
Gaussian curvature at least $1$.
By the
monotonicity formula, $R^{-2}A_M(R)$ is bounded as
$R\to 0$. Geometric measure theory implies that any
sequence of expansions of $M$ converges (up to a
subsequence) to a minimal cone over a configuration of
geodesic arcs in $\esf^2(1)$. Since any smooth point of
such a minimal cone is flat, we contradict the
 existence of the non-flat minimal leaf $L_1$.
This finishes the proof of Case {\sc I.}
\par
\vspace{.2cm}
{\sc Case II:} {\bf Suppose that $N=\R^3$, $p=\vec{0}$
and ${\cal L}$ consists of a possibly disconnected,
 properly embedded minimal surface in
 $\overline{\B}(r)-\{ \vec{0}\} $.}
\par
\vspace{.2cm}
\noindent
Consider the intersection of ${\cal L}$ with the
closed ball of radius $r'\in (0,r)$. We claim that
every leaf $L$ of ${\cal L}$
having $\vec{0}$ in its closure intersects
$\partial \overline{\B}(r')$; otherwise, the supremum of the
function $R|_L$ is some $r_1\leq r'$ and there exists a leaf
$L_1$ in the closure of $L$ that lies at the inner side
of $\esf ^2(r_1)$ and touches this sphere at some point,
which contradicts the mean comparison principle for $L_1$ and
$\esf^2(r_1)$ and proves our claim.
 Since $\partial \overline{\B}(r')$ is compact, we conclude that
 there are at most a finite number of leaves of ${\cal L}$
having $\vec{0}$ in its closure. If there are
two leaves of ${\cal L}\cap \overline{\B}(r')$
which have $\vec{0}$ in their closure, then each of
these leaves extends smoothly across the origin
by the previous Case I, and we contradict the
maximum principle for minimal surfaces. Therefore,
at most one leaf of ${\cal L}$ has the origin in its closure,
and the other components, which are compact, do not
intersect a certain ball $\overline{\B }(r')$ for some
$r'\in (0,r)$ small enough. Hence, ${\cal L}$
extends in this case.
\par
\vspace{.2cm}
{\sc Case III:} {\bf Suppose that $N=\R^3$, $p=\vec{0}$
and ${\cal L}$ is a minimal lamination which
does not intersect any small punctured neighborhood of
$\vec{0}$ in a properly embedded surface} (note that
Cases I--III finish the
 $\rth$-setting of Theorem~\ref{tt2}).
\par
\vspace{.2cm}
\noindent
In this case, every punctured neighborhood of
$\vec{0}$ intersects a limit leaf of~${\cal L}$.
Since the set of limit leaves of ${\cal L}$ is closed,
it follows that  ${\cal L}$ contains a limit leaf $F$
with $\vec{0}$ in the closure of $F$.

We claim that any blow-up limit of ${\cal L}$ from
$\vec{0}$ converges outside $\vec{0}$ to a flat
 lamination of $\R^3$ by planes. Since
 $|K_{\cal L}|R^2$ is scale-invariant, our claim
follows by proving that for any $\ve >0$, there is
$r(\ve )\in (0,r)$ such that $|K_{\cal L}|R^2<\ve $ on
${\cal L}\cap \B (r(\ve ))$. Arguing by contradiction,
suppose there exists a sequence of points
$q_n\in {\cal L}$ converging to $\vec{0}$ such that
$|K_{\cal L}|(q_n)|q_n|^2$ is bounded away from zero.
Then, after expansion of ${\cal L}$ by
$\frac{1}{|q_n|}$ and taking a subsequence, we obtain a
limit which is a non-flat minimal lamination
${\cal L}_1$ of $\R^3-\{ \vec{0}\}$ that satisfies the
hypotheses in Proposition~\ref{propos1}. In
particular, ${\cal L}_1$ does not contain flat leaves.
The limit leaf $F$ in ${\cal L}$ produces under
expansion a leaf $F_1$ (which is stable
since $F$ is stable) in ${\cal L}_1$. Since $F_1$ is
complete outside the origin, the stability lemma implies
that $F_1$ is a plane, which is a contradiction. Now our claim is proved.

By the above claim, we know that any blow-up limit of
${\cal L}$ is a minimal lamination of $\R^3-
\{ \vec{0} \}$ by parallel planes. It follows that for
$\ve > 0$ sufficiently small, in the annular domain
$A = \{ x \in \rth \mid \frac{1}{2} \leq |x| \leq 2 \}$
the normal vectors to the leaves of ${\cal L}_\ve =
(\frac{1}{\ve} {\cal L}) \cap A$ are almost parallel,
and after a rotation (which might depend on $\ve $),
we will assume that the unit normal
vector to the leaves of ${\cal L}_{\ve }$ lies in a small
neighborhood of $\{ \pm (0,0,1)\} $.
Hence, for such a sufficiently small $\ve$, each
 component $C$ of ${\cal L}_\ve$ that intersects
${\esf}^2(1)$ is of one of the following four types, see
Figure~\ref{fig4new}:
\begin{figure}
\begin{center}
\includegraphics[width=14cm,height=6.7cm]{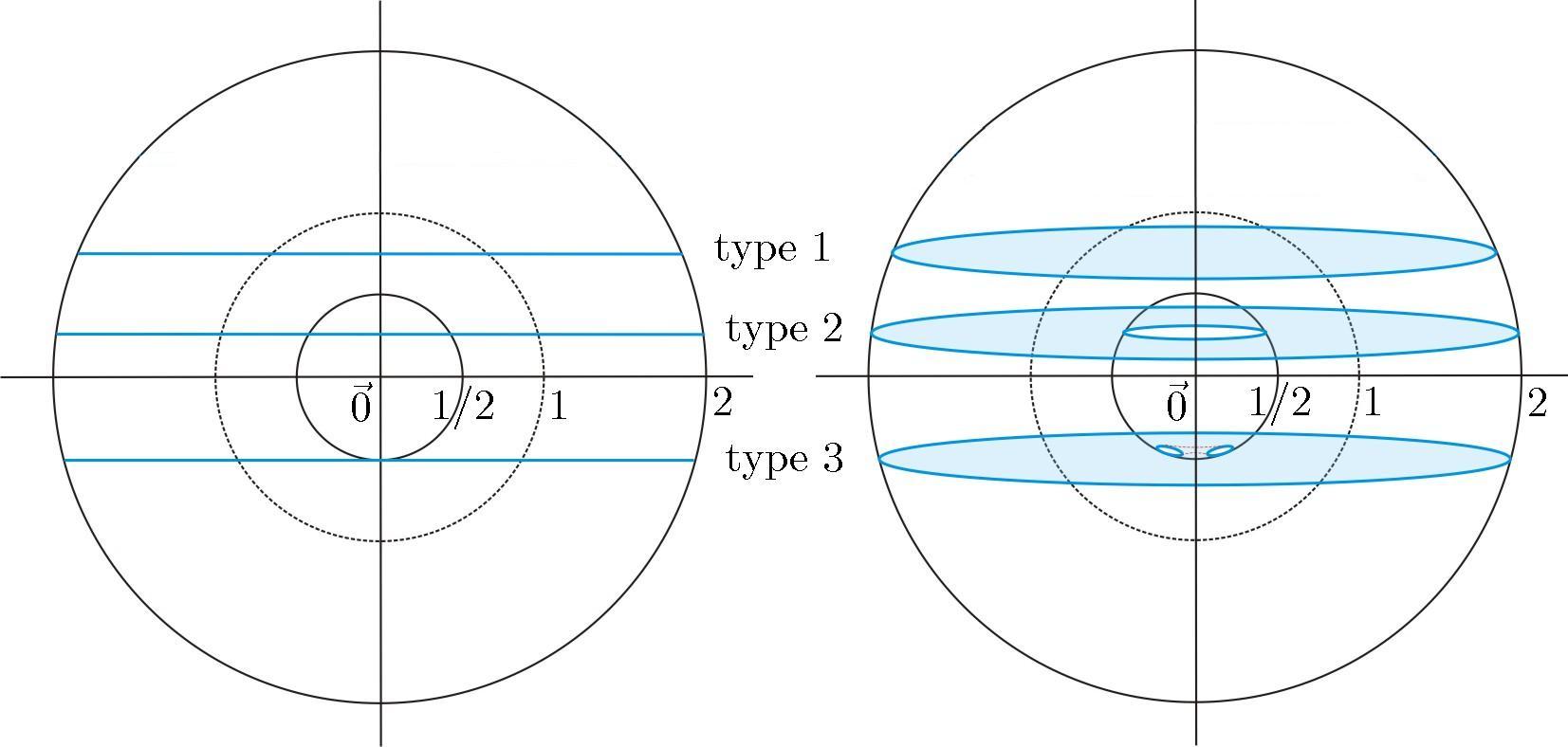}
\caption{Type 1, 2, 3 connected components of ${\cal L}_{\ve }$.}
\label{fig4new}
\end{center}
\end{figure}
\begin{enumerate}
\item A compact disk with boundary in
${\esf}^2(2)$;

\item A compact annulus with one boundary curve in
${\esf}^2 (\frac{1}{2})$ and the other boundary
     curve in ${\esf}^2 (2)$;

\item A compact planar domain whose boundary
 consists of a single closed curve in
 ${\esf}^2 (2)$ together with at least two closed
 curves in ${\esf}^2 (\frac{1}{2})$,  and
where the outer boundary curve bounds a compact disk in
$\frac{1}{\ve} {\cal L}$;

\item An infinite multigraph $\mathcal{M}$ whose limit set
 consists of two compact components of ${\cal L}_{\ve }$
of type~2. To see why this is the only possibility for
the limit set lim$(\mathcal{M})$ of $\mathcal{M}$,
note that lim$(\mathcal{M})$ cannot contain a component of type 1
by an elementary covering argument. Also, lim$(\mathcal{M})$
can be supposed not to contain components of type 3,
by choosing $\ve $ smaller.
\end{enumerate}

One consequence of the description above is that if a component
$\mathcal{M}$ of ${\cal L}_{\ve }$ of type~4 occurs, then $\mathcal{M}$ cannot intersect the
complement of the slab $\{ |x_3|\leq \frac{1}{2}\} $ in $A$. In fact,
by taking $\ve >0$ small enough, we conclude that spiraling multigraph
components of ${\cal L}_{\ve }$ can only occur in the intersection of $A$
with an open slab $\Delta $ of small width around height $0$.
Note that the intersection of a spiraling component $\mathcal{M}$ with $\esf^2(2)$
is an embedded spiraling curve
$\G (2)$ contained in $\Delta $ that limits to two closed, pairwise disjoint
curves $C_1(2),C_2(2)\subset \esf^2(2)\cap \overline{\Delta }$ which are almost horizontal,
and the same description holds for the intersection
of $\mathcal{M}$ with $\esf^2(\frac{1}{2})$, defining an embedded spiraling curve
$\G (\frac{1}{2})$ contained in $\esf^2(\frac{1}{2})\cap \Delta $
that limits to two closed, pairwise disjoint, almost horizontal
curves $C_1(\frac{1}{2}),C_2(\frac{1}{2})\subset \esf^2(\frac{1}{2})\cap \overline{\Delta }$.
In fact, this description also holds for the intersection of $\mathcal{M}$
with every intermediate sphere $\esf^2(\tau )$, $\frac{1}{2}\leq \tau \leq 2$, and the
union of the related closed curves $\cup _{\tau \in [\frac{1}{2},2]}C_j(\tau )$, $j=1,2$,
defines the two compact components of ${\cal L}_{\ve }$ of type 2 in the limit set of $\mathcal{M}$.

Next we check that type~4 components of
${\cal L}_{\ve }$ cannot occur for $\ve >0$
sufficiently small: if for some sufficiently small
$\ve_0 $, $ {\cal L}_{\ve{_0}}$ has a component of type 4,
then this multigraph component persists for all
$\ve \in (0,\ve_0)$, varying in a continuous manner
in terms of $\ve$, as well as the two annular components of type 2 which are the limit
set of this multigraph component. Thus, the existence of a multigraph
component in ${\cal L}_{\ve_{0}}$ implies that in the
original scale, ${\cal L}\cap \B (2\ve_0)$ has two
properly embedded  annular leaves in
$\B (2\ve_0) - \{ \vec{0} \}$. By our previously
considered Case I in this proof, these two annular leaves extend
smoothly to two minimal disks that only intersect at the
origin, thereby contradicting the maximum principle
for minimal surfaces. This contradiction shows that
only components of types 1, 2, 3 can occur in
${\cal L}_{\ve}$ for $\ve $ small.

Recall that $F$ is a limit leaf of ${\cal L}$ with $\vec{0}$ in its
closure. We claim that there exists $r_1\in (0,r)$ such that
$F\cap [\overline{\B }(r_1)-\{ \vec{0}\} ]$
is a proper annulus in $\overline{\B }(r_1)-\{ \vec{0}\} $.
This follows from the fact that for any $\ve >0$ small enough,
$F$ produces a type 2 component of ${\cal L}_{\ve }$ (type 4 is
discarded by the previous paragraph, and types 1, 3 can be discarded
by taking $\ve $ small enough). Since $F\cap [\overline{\B }(r_1)-\{ \vec{0}\} ]$
is a proper annulus, we can apply the already proven Case~I
to $F\cap [\overline{\B }(r_1)-\{ \vec{0}\} ]$ and conclude that
$F\cap [\overline{\B }(r_1)-\{ \vec{0}\} ]$ extends smoothly across $\vec{0}$ and
$\overline{F}$ is a compact minimal disk.

Next we claim that for $\ve >0$ small enough and for every leaf component of
${\cal L}\cap \overline{\B }(\ve )$ the tangent planes to this leaf component
at any of its points make an angle less
than $\pi /4$ with the tangent plane $T_{\vec{0}}\overline{F}$.
Otherwise, we find a sequence of points $p_n\in {\cal L}$
converging to $\vec{0}$ as $n\to \infty $, such that the leaves $L_n$ of
${\cal L}$ passing through $p_n$ have tangent planes $T_{p_n}L_n$ making
an angle larger than $\pi /4$ with $T_{\vec{0}}\overline{F}$. After rescaling
${\cal L}$ by the homothety with factor $1/|p_n|$ centered at $\vec{0}$, we produce a
sequence of laminations with the same quadratic decay constant. After passing to a subsequence,
this sequence converges to a family of parallel planes, which is impossible since one of the planes is
$T_{\vec{0}}\overline{F}$ and another one is a plane passing through a point
in $\esf^2(1)$ making an angle greater than $\pi /4$ with $T_{\vec{0}}\overline{F}$.
This contradiction proves our claim.

By the claim in the last paragraph, the inner product of the unit normal vector to
any leaf component of ${\cal L}\cap \overline{\B }(\ve )$ with the normal direction
to $T_{\vec{0}}\overline{F}$ defines a Jacobi function with constant sign on
such a leaf component. Hence, ${\cal L}\cap \overline{\B }(\ve )$ consists of
a collection of stable surfaces. Finally, curvature estimates for stable minimal surfaces
away from their boundaries (Schoen~\cite{sc3})
gives the desired lamination structure for ${\cal L}$ around $\vec{0}$. This
%
 finishes the proof in this Case III.
\par
\vspace{.2cm}
{\sc Case IV:} {\bf Suppose $N$ is a Riemannian manifold,
$p\in N$ and for some $r'\in (0,r)$
${\cal L} \cap B_N(p,r')$ is a non-compact, possibly disconnected,
properly embedded minimal surface $M$ in $B_N(p,r') - \{ p \}$.}
\par
\vspace{.2cm}
\noindent
In this case, $\exp _p$ yields $\R^3$-coordinates on
$B_N(p,r')$ centered at $p \equiv \vec{0}$, for $r'>0$
small enough.
It follows from Theorems~3.1 and 4.1 in~\cite{hl1} that
$M$ is a locally rectifiable stationary current
(relative to its boundary) with finite area.
Hence, under homothetic expansions of coordinates, $M$
has minimal limit tangent cones in $\R^3$
at $\vec{0}$. 

If there exists an $\ve > 0$ and a sequence
$\{ p_n \}_n \subset M$ converging to $p$ such that
$\ve \leq |\sigma _M|(p_n) d_N(p_n,p)$ for all $n$,
then a subsequence of the expanded surfaces
$\widetilde{M}_n = \frac
{1}{d_N(p_n,p)} M$ in $\frac {1}{d_N(p_n,p)} B_N(p,r')$
converges to a non-flat minimal lamination
${\cal L}_{\infty}$ of $\rth-\{ \vec{0}\} $. Since
${\cal L}_{\infty}$ is not flat at some point of
$\esf^2(1)$, then ${\cal L}_{\infty }$ has a leaf which is not a cone,
which contradicts the conclusion of the
previous paragraph.

Hence, any sequence of
homothetic blow-ups of $M$ has a subsequence which
converges (possibly with finite multiplicity) to a minimal lamination
of $\R^3-\{ \vec{0}\} $ by parallel planes. Also, the
arguments in the first paragraph of the proof of
Lemma~\ref{lemma2} imply that
$M$ has a finite number of annular ends.
Furthermore, the fact that under expansions of $M$
we obtain minimal cones in $\R^3-\{ \vec{0}\} $,
 implies that
$M$ intersects almost orthogonally the geodesic
spheres $\partial B_N(p,r'')$ for all
$r''>0$ sufficiently small. Hence, each of the annular
ends of $M$ has linear area growth with respect to
the complete metric $\frac{1}{d_N(p,\cdot )^2} \langle, \rangle$
(here $\langle ,\rangle $ denotes the original metric on $N$).
Therefore, the ends of $M$ are conformally
punctured disks. Standard regularity theory implies
 that the conformal harmonic map from each of these
 annular ends of $M$ into $N$
extends smoothly across the punctured disks to
a conformal harmonic map (see Gr\"{u}ter~\cite{gr1}),
hence to a branched minimal
immersion into $B_N(p,r'')$ as well. Such a branched minimal
immersion is free of branch points, since $M$ is
embedded. Finally, the maximum principle for minimal
surfaces implies that $M$ has only one annular end.
This finishes the proof of Case IV.

To conclude the
proof of the theorem, it remains to solve the following:
\par
\vspace{.2cm}
{\sc Case V:} {\bf Suppose $N$ is a Riemannian manifold,
$p\in N$ and for all $r'\in (0,r)$,
$ {\cal L} \cap (B_N(p,r') - \{ p \})$ is not a
properly embedded minimal surface in
$B_N(p,r') - \{ p \}$.
}
\par
\vspace{.2cm}
\noindent
In this case ${\cal L}$ contains a limit
leaf $L$ with $p \in \overline{L}$. Since the proof of
this case is very similar to the proof of Case III,
we will only comment on the differences. The same arguments
as in Case~III prove that every blow-up limit of ${\cal L}$
from $p$ converges outside $\vec{0}\in \R^3$ to a
lamination of $\R^3$ by planes (only exchange
$|K_{\cal L}|R^2$ by $|\sigma _{\cal L}|\, d_N(p,\cdot )$).
Then, for $\ve >0$ sufficiently small the normal vectors to
the leaves of ${\cal L}_{\ve }=\frac{1}{\ve }(
{\cal L}\cap \{ x\in B_N(p,r')\ | \ \frac{\ve }{2}\leq
\, d_N(p,x)\leq 2\ve \} )$ at any of their points
are almost parallel, where we
are using $\R^3$-coordinates in $B_N(p,r')$ via the
exponential map $\exp _p$ of $N$ for $r'\in (0,r)$
small enough. The components of ${\cal L}_{\ve }$ are
of one of the types 1--4 in Case III (we only exchange
$\esf^2(R)$ by $\frac{1}{\ve }\partial B_N(p,\ve R)$,
where $R=\frac{1}{2}$ or $R=2$). Type 4 components
cannot occur if $\ve >0$ is sufficiently small, by the
same arguments as in Case III (such arguments rest on
the validity of Case II, which we now substitute by
Case IV). By the same arguments as in Case~III, the
limit leaf $L$ is a proper annulus in $B_N(p,r)$ that extends
smoothly across $p$, and every blow-up limit of ${\cal L}$
from $p$ converges outside $\vec{0}\in \R^3$ to a
lamination of $\R^3$ by planes parallel to the tangent
plane $T_p\overline{L}$ to the  extended surface $\overline{L}=L\cup \{ p\} $.

Consider a small geodesic disk $D_{\overline{L}}(p,\de)$ centered at $p$ with radius $\de $
in $\overline{L}$, and let $\eta $ be the unit normal vector field
 to $D_{\overline{L}}(p,\de )$ in $N$.
 Pick coordinates $q=(x,y)$ in $D_{\overline{L}}(p,\de)$
and let $t\in [-\tau ,\tau ]\mapsto
\g _q(t)$ be the unit speed geodesic of $N$ with initial conditions $\g _q(0)=q$,
$\g_q'(0)=\eta (q)$
(here $\tau >0$ can be taken independent of $q\in D_{\overline{L}}(p,\de )$).
Then for some $\tau >0$ small,
$(x,y,t)$ produces ``cylindrical'' normal coordinates in a neighborhood $U$ of $p$ in $N$,
and we can consider the natural projection
\[
\Pi \colon U\to D_{\overline{L}}(p,\de ),\ \Pi (x,y,t)=(x,y).
\]
Since every blow-up limit of ${\cal L}$ from $p$ converges outside $\vec{0}\in \R^3$ to a
lamination of $\R^3$ by planes parallel to $T_{p}\overline{L}$,
we conclude that for $\de $ and $\tau $ sufficiently small,
the angle of the intersection of any leaf component $L_U$ of ${\cal L}\cap U$
with any geodesic $\g _q$ as above can be made arbitrarily close to $\pi /2$. Taking
$\de $ much smaller than $\tau $, a monodromy argument implies any leaf component $L_U$ of ${\cal L}
\cap U$ which contains a point at distance at most $\de /2$ from $p$ is a graph over $D_{\overline{L}}(p,\de )$
(in other words, $\Pi $ restricts to $L_U$ as a diffeomorphism onto $D_{\overline{L}}(p,\de )$).
We will finish this Case~V by proving that
this graphical property implies the desired uniform bound for the second fundamental form
$\sigma _{\cal L}$ of ${\cal L}$ around $p$: otherwise there exists a sequence of points $p_n$ in leaves
$L_n$ of ${\cal L}$ converging to $p$, such that $|\sigma _{L_n}|(p_n)$ diverges.
Without loss of generality, we can assume that $p_n\in U$ and $p_n$ is a point where
the following function attains its maximum:
\[
f_n=|\sigma _{L_n}|\, d_{\overline{L}}(\Pi (\cdot ),\partial D_{\overline{L}}(p,\de ))
\colon L_n\cap U\to [0,\infty ),
\]
where $d_{\overline{L}}$ denotes the intrinsic distance in $\overline{L}$ to the boundary
$\partial D_{\overline{L}}(p,\de )$.
 Now expand the above coordinates $(x,y,z)$ centered at $p_n$ (so that $p_n$
becomes the origin) by the factor $|\sigma _{L_n}|(p_n)\to \infty $. Under this expansion, $U$ converges
to $\R^3$ with its usual flat metric and the geodesics $\g _q$
converge to parallel lines (the canonical coordinates $(x,y,z)$ in $\R^3$ are not necessarily
those coming from the $(x,y,t)$-coordinates in $U$).
The graphical property that $\Pi $ restricts to $L_U$ as a diffeomorphism
onto $D_{\overline{L}}(p,\de )$ gives that after passing to a subsequence, the minimal graphs $L_n\cap U$
converge to a minimal surface in $\R^3$ which is a entire graph. By the Bernstein Theorem, such
a limit is a flat plane. This contradicts that the homothetic expansion factors coincide
with the norms of the second fundamental form of $L_n\cap U$ at $p_n$ for all $n$. This contradiction finishes the proof of Case V,
and thus Theorem~\ref{tt2} is proved.
%
%
%
%
{\hfill\penalty10000\raisebox{-.09em}{$\Box$}\par\medskip}

Theorem~\ref{tt2} supports the conjecture that
a properly embedded minimal surface in a punctured ball
extends smoothly through the puncture. An important
partial result for this conjecture was obtained by
Gulliver and Lawson~\cite{gl1}, who proved it in the special
case that the surface is stable (note that
 Theorem~\ref{tt2} generalizes the result by Gulliver
 and Lawson, by curvature estimates for stable minimal surfaces~\cite{sc3}).
  This isolated singularity conjecture is one of the fundamental open
problems in minimal surface theory, and it is a special case
of the following more general conjecture:

\begin{conjecture}[Fundamental Singularity
Conjecture] \label{conjlamination} Suppose ${\cal S} \subset \rth$
is a closed set whose one-dimensional Hausdorff measure is zero. If
$\lc$ is a minimal lamination of $\rth - {\cal S}$, then ${\cal L}$
extends across ${\cal S}$ to a minimal lamination of $\rth$.
\end{conjecture}
Since the union of a catenoid with a plane passing through its waist
circle is a singular minimal lamination of $\rth$ whose singular set
is the intersecting circle, the above conjecture represents the
strongest possible conjecture.
We point out to the reader that Conjecture~\ref{conjlamination} has
a global nature, because there exist interesting minimal laminations
of the open unit ball in $\rth$ punctured at the origin which do not
extend across the origin, see Section~\ref{secex}.  In hyperbolic
three-space $\HH^3$, there are rotationally invariant global minimal
laminations which have a similar unique isolated singularity. The
existence of these global singular minimal laminations of $\HH^3$
demonstrate that the validity of Conjecture~\ref{conjlamination}
must depend on the metric properties of $\rth$.

\section{The characterization of minimal surfaces with quadratic decay of
curvature.}
\label{sec6}
In this section we will prove Theorem~\ref{thm1introd}
stated in the introduction.
\begin{proposition}
\label{proposlamin} Let ${\cal L}$ be a non-flat minimal lamination
of $\R^3-\{ \vec{0}\} $ with quadratic decay of curvature. Then,
${\cal L}$ consists of a single leaf, which extends to a connected,
properly embedded minimal surface in $\R^3$.
\end{proposition}
\begin{proof}
By Proposition~\ref{propos1}, each leaf $L$ of ${\cal L}$ is a
minimal surface which is properly embedded in $\R^3-\{ \vec{0}\} $.
Applying Theorem~\ref{tt2} to each $L\in {\cal L}$, we
deduce that $L$ extends to a properly embedded minimal surface in
$\R^3$. Finally, ${\cal L}$ consists of a single leaf by the maximum
principle applied at the origin and the Strong Half-space
Theorem~\cite{hm10}.
\end{proof}

\begin{theorem}
\label{thm2} Let $M\subset \R^3$ be a complete, embedded, non-flat
minimal surface with compact boundary (possibly empty). If $M$ has
quadratic decay of curvature, then $M$ is properly embedded in
$\R^3$ with finite total curvature.
\end{theorem}
\begin{proof}
We first check that $M$ is proper when $\partial M$ is empty.
In this case where $\partial M=\mbox{\O }$, the closure of $M$ in $\R^3-\{ \vec{0}\} $ is
a minimal lamination of $\R^3-\{ \vec{0}\} $ satisfying the
conditions in Proposition~\ref{proposlamin}. It follows that $M$ is
a properly embedded minimal surface in $\R^3$ with bounded
curvature.

We now prove that $M$ is also proper when $\partial M \neq \mbox{
\O}$. Since $\partial M$ is compact, we may assume $\vec{0}\notin
\partial M$ by removing a compact subset from $M$. Therefore, there
exists an $\ve >0$ such that $\partial M\subset \R^3-\B (\ve )$.
Thus, Theorem~\ref{tt2} gives that $\overline{M}\cap (\B
(\ve )-\{ \vec{0}\} )$ has bounded curvature, and so, $M$ does as
well (in order to apply Theorem~\ref{tt2} we need $M \cap
(\B (\ve )-\{ \vec{0}\} )$ to be non-empty; but otherwise $M$ would
have bounded curvature so we would arrive to the same conclusion).
If $M$ were not proper in $\R^3$, then $\overline{M}-\partial M$ has
the structure of a minimal lamination of $\R^3-\partial M$ with a
limit leaf $L$ which can be assumed to be disjoint from $M$
(otherwise $M$ is stable with compact boundary, hence $M$ has
finite total curvature by Fischer-Colbrie~\cite{fi1} and thus, $M$ is proper).
Since we may also assume,
after possibly removing an intrinsic neighborhood of $\partial M$,
that $\overline{L}\cap\partial M=\mbox{\O }$, then  $L$ is complete
and stable, and hence, $L$ is a plane. Since $M$ limits to $L$ and
$M$ has bounded curvature, we obtain a contradiction by applying the arguments
in the proof of Lemma~1.3 in Meeks and Rosenberg~\cite{mr8}.
Hence, $M$ is proper regardless of whether or not $\partial M$ is empty.

From now on, we will assume that $M$ is non-compact and properly
embedded in $\R^3$. Since $\partial M$ is compact (possibly empty),
there exists an $R_1>0$ such that $\partial M\subset \B (R_1)$. It
remains to show that $M$ has finite total curvature.

Let $C_1\in (0,1)$ be the constant given by the
statement of Lemma~\ref{lemma2}.
Suppose first that there exists $R_2>R_1$ such that
    $|K_M|R^2\leq C_1$ in $M-\B (R_2)$.
Applying Lemma~\ref{lemma2} to each component of $M-\B (R_2)$, such
components are annular ends with finite total curvature. Since $M$
is proper, there are a finite number of such components as $M\cap
\esf^2(R_2)$ is compact. Thus, $M$  has finite total curvature, which
proves the theorem in this case.

Now assume that there exists a sequence $\{ p_n\} _n\subset M$
diverging to $\infty $ such that $C_1\leq |K_M|(p_n)|p_n|^2$ for all
$n$, and we will find a contradiction. The homothetically shrunk
surfaces $\widetilde{M}_n=\frac{1}{|p_n|}M$ also have quadratic curvature
decay and their boundaries collapse to $\vec{0}$.
Thus, after choosing a subsequence, we may assume that the $\widetilde{M}_n$
converge to a minimal lamination ${\cal L}$ of $\R^3- \{ \vec{0}\} $, and
$|K_{\cal L}|R^2$ also  decays quadratically. Since
$|K_{\widetilde{M}_n}|(\frac{1}{|p_n|} {p_n})\geq C_1$ and we can
assume $\frac{1}{|p_n|}{p_n}\to \widetilde{p}_{\infty }\in
\esf^2(1)$, there exists a non-flat leaf $L\in {\cal L}$ with
$\widetilde{p}_{\infty }\in L$. By
Proposition~\ref{proposlamin}, ${\cal L}=\{ L\} $ and $\overline{L}$
is properly embedded in $\R^3$.
If the convergence of the
$\widetilde{M}_n$ to ${\cal L}$ had multiplicity greater than one,
then $L$ would be flat (see Corollary~\ref{sc2}), but it
is not. Also note that $\overline{L}$ is connected, and so,
it must pass through the origin. Since $\overline{L}$ is properly embedded,
the multiplicity of the limit $\widetilde{M}_n
\to L$ is one and $\vec{0}\in \overline{L}$, then we have $\lim
_{r\to 0}r^{-2}\mbox{Area}(\overline{L}\cap \B (r))=\pi $
(i.e. the density of $\overline{L}$ at the origin is
1) and there exists $\ve >0$ such that for all $\ve '\in (0,\ve ]$,
$\overline{L}\cap \overline{\B }(\ve ')$ consists of a non-flat disk passing
through the origin and transverse to $\esf^2(\ve ')$ along its boundary.

We will study the function
\[
R\in [R_1,\infty )\mapsto f(R)=R^{-2}\mbox{Area}
(M\cap [\B (R)-\B (R_1)]),
\]
which is non-decreasing by the monotonicity formula.
By equation (\ref{Stokes}), we have
\begin{equation}
\label{Stokes1}
f(\ve |p_n|)=\frac{1}{2\ve ^2|p_n|^2}\int _{\partial M}\langle p,\nu \rangle
+ \frac{1}{2\ve ^2|p_n|^2}\int _{M\cap \esf^2(\ve |p_n|)}\langle p,\nu \rangle ,
\end{equation}
where $\nu $ is the unit exterior conormal vector to $M\cap \B (\ve |p_n|)$
along its boundary. Changing variables in the second integral in~(\ref{Stokes1}) we have
\[
f(\ve |p_n|)=\frac{1}{2\ve ^2|p_n|^2}\int _{\partial M}\langle p,\nu \rangle +
\frac{1}{2\ve ^2}\int _{(\frac{1}{|p_n|}M)\cap \esf^2(\ve )}\langle q,\nu \rangle .
\]
Since the sequence $\{ \frac{1}{|p_n|}M\} _n$ converges to $L$ with multiplicity one
on compact subsets of $\R^3-\{ \vec{0}\} $, then
\begin{equation}
\label{eq:limf}
\lim _{n\to \infty }f(\ve |p_n|)=\frac{1}{2\ve ^2}\int _{L\cap \esf^2(\ve )}\langle q,\nu \rangle
=\frac{1}{\ve ^2}\mbox{Area}[L\cap \B (\ve )].
\end{equation}
Since $f$ is monotonically non-decreasing, we conclude from (\ref{eq:limf}) that
$\lim _{R\to \infty }f(R)$ exists, and equals $l(\ve)=\frac{1}{\ve ^2}\mbox{Area}[L\cap \B (\ve )]$.
This is a contradiction, since $l(\ve )$ is strictly increasing as a function of $\ve $ small
as $L$ is not flat.
 This contradiction proves the theorem.
\end{proof}

\begin{corollary}
\label{corollamin} Let ${\cal L}$ be a non-flat minimal lamination
of $\R^3-\{ \vec{0}\} $. If ${\cal L}$ has quadratic decay of
curvature, then ${\cal L}$ consists of a single leaf, which extends
to a properly embedded minimal surface with finite total curvature
in $\R^3$.
\end{corollary}
\begin{proof}
This follows easily from Proposition~\ref{proposlamin} and Theorem~\ref{thm2}.
\end{proof}

Theorem~\ref{thm1introd} follows immediately from
Theorem~\ref{thm2}. We just remark that the last statement in
Theorem~\ref{thm1introd} follows from the finite total curvature
property, since a non-flat, complete, embedded, non-compact
minimal surface of finite total curvature has a positive number of
catenoidal ends and possibly finitely many planar ends. A simple
calculation shows that the growth constant $C^2$ in
Theorem~\ref{thm1introd} depends on the maximum logarithmic growth
$C$ of the catenoidal ends of $M$.

\begin{remark}
{\rm Given $C>0$, let ${\cal F}_C$ denote the
family of all complete, embedded, connected minimal surfaces
$M\subset \R^3$ with quadratic curvature decay constant $C$,
normalized so that the maximum of the function $|K_M|R^2$ occurs at
a point of $M\cap \esf^2(1)$. In \cite{mpr19} we  applied
Theorem~\ref{thm1introd} to prove that ${\cal F}_C$ is naturally a
compact metric space and that for $C$ fixed, there is a bound on the
genus and number of ends of all surfaces in ${\cal F}_C$ and that the subsets of ${\cal
F}_C$ with fixed topology are compact.}
\end{remark}


\section{Minimal surfaces and minimal laminations with countably many singularities.}
\label{sec10}

In Theorem 1 of~\cite{mpr18} (see also~\cite{mpr19})
we proved that the sublamination
$\mbox{\rm Lim}(\lc) $ of limit leaves of a minimal
 lamination $\lc$ of a three-manifold $N$ consists of stable minimal surfaces,
 and more strongly, their two-sided covers are stable.
 An immediate consequence of this result is that
the set $\mbox{\rm Stab}(\lc)$ of
stable leaves of $\lc$ is a sublamination of $\lc$ with $\mbox{\rm
Lim}(\lc) \subset \mbox{\rm Stab}(\lc)$.  Using
these observations
together with standard curvature estimates~\cite{mpr19,rst1,sc3} for stable
two-sided minimal surfaces away from their boundaries, we will demonstrate the
following consequence of Theorem~\ref{tt2}.
\begin{corollary}
\label{corrs}
Suppose that $N$ is a not necessarily complete Riemannian
three-manifold. If $W\subset N$ is a closed countable subset and
${\cal L}$ is a minimal lamination of $N-W$, then the closure of any
collection of its stable leaves extends across $W$ to a minimal
lamination of $N$ consisting of stable minimal surfaces. In
particular:
\begin{enumerate}
\item The closure $\overline{\mbox{\rm Stab}(\lc)}$ in $N$ of
the collection of stable leaves of $\lc$ is a minimal lamination
of $N$ whose leaves are stable minimal surfaces.
\item  The closure $\overline{\mbox{\rm Lim}(\lc)}$ in $N$ of the collection
of limit leaves of ${\cal L}$ is a minimal lamination
of $N$.
\item If ${\cal L}$ is a minimal foliation of $N-W$, then ${\cal L}$
extends across $W$ to a minimal foliation of~$N$.
\end{enumerate}
\end{corollary}
\begin{proof}
We start by proving the first statement in Corollary~\ref{corrs}.
Since the extension of the closure $\lc _1$ of any given collection
of stable leaves in ${\cal L}$ is a local question,
it suffices to prove the corollary in small, open extrinsic balls in $N$. As
$W$ is countable, we can take these balls so that each of their boundaries
are disjoint from $W$, and their closures in $N$ are compact. It follows
that for every such ball $B_N$, the set $W\cap B_N$ is a complete
countable metric space. By Baire's theorem, the set of isolated
points in $W\cap B_N$ is dense in $W\cap B_N$. Curvature
estimates for stable minimal surfaces~\cite{ros9,rst1,sc3}
together with Theorem~\ref{tt2} imply that ${\cal L}_1 \cap B_N$ extends across every isolated point
of $W \cap B_N$ to a minimal lamination of $B_N-W'$, where $W'$ is the subset of
non-isolated points in $W\cap B_N$. Consider the minimal closed subset $W''$ of $W\cap B_N$
(under inclusion) such that ${\cal L}_1$ does not extend across any
point $p\in W''$ to a minimal lamination. We want to prove that $W''=\mbox{\O }$. Arguing by
contradiction, suppose $W''\neq \mbox{\O }$. As before, $W''$ is a
countable complete metric space and so, Baire's theorem insures that
the set of its isolated points is dense in $W''$. But the lamination
${\cal L}_2$ obtained by extension of ${\cal L}_1$ across $W-W''$,
extends through every isolated point of $W''$ by the above
arguments; hence we contradict the minimality of $W''$ and the first
statement in Corollary~\ref{corrs} is proved.

Note that items~{\it 1, 2, 3} of this corollary follow directly from the already proven first
statement of the corollary and from the following facts:
\begin{itemize}
\item Limit leaves of a minimal lamination are stable~\cite{mpr19}.
\item A smooth limit of limit leaves of a lamination is also a limit leaf.
\item Every leaf of a foliation is a limit leaf.
\end{itemize}
Hence, the proof is complete.
\end{proof}
\par
\vspace{.5cm}

As an application of Corollary~\ref{corrs}, we have the following
generalization of the stability lemma (Lemma~\ref{lema1}) in the minimal case. The proof
of the next result follows immediately from the fact that every
embedded, stable minimal surface in a Riemannian three-manifold has
local curvature estimates, and so, its closure has the structure of
a minimal lamination all whose leaves are stable.

\begin{corollary}
\label{corolcountsing}
Let $M$ be a connected, embedded, stable minimal surface
in a Riemannian three-manifold $N$. Suppose that
$M$ is complete outside a countable closed set
${\cal S}$ of $N$.
Then, the closure of $M$ has the structure of a
minimal lamination of $N$, and the intrinsic metric
completion of $M$ is a leaf of this lamination.
In particular, if $N=\R^3$, then the closure of $M$ is a plane.
\end{corollary}

\section{Appendix: An alternative proof of Proposition~\ref{propos1}.}
\label{appendix}

The main technical result of this paper, Proposition~\ref{propos1}, was proved
as a consequence of Lemmas~\ref{lema2new} and~\ref{lema1new}, together with an
analysis of type I and type II curves. In this analysis we used flux arguments,
one of which was based in Colding-Minicozzi theory. We next explain how to rule out
both cases with a different argument, based on conformal properties
for minimal surfaces. The arguments that follow will use the same notation and results
in Section~\ref{sec4}, that will be assumed to hold up to Section~\ref{subsect1}.

Consider the metric $\widehat{g}=\frac{1}{R^2}\langle ,\rangle $ on
$\R^3-\{ \vec{0}\} $. 
 As we noticed in the proof of Lemma~\ref{lema1}, $(\R^3-\{ \vec{0}\}
,\widehat{g})$ is isometric to $\esf^2(1)\times \R $ endowed with its
product metric $\wt{g}$. In fact, the map
\begin{equation}
\label{mapF}
F\colon (\R^3-\{ \vec{0}\} ,\widehat{g}) \to (\esf^2(1)\times \R ,\wt{g}), \quad
F(p)=\left( \frac{p}{R(p)},\log R(p)\right) ,
\end{equation}
is an isometry. If $C$ is any
non-negative vertical half-cone $C$ minus its vertex at $\vec{0}$,
then $A_C:=F(C)$ is the flat cylinder $(C\cap  \esf ^2(1))\times \R $.
In the particular case where $C$ is
the $(x_1,x_2)$-plane, then $A_C=\esf ^1(1)\times \R$
is totally geodesic. We endow $\esf
^2(1)\times \R $ with global coordinates $(\varphi ,\theta ,t)$ so
that $(\varphi ,\theta )$ are the natural spherical coordinates on
$\esf ^2(1)$ and $t$ denotes the vertical linear coordinate in $\esf
^2(1)\times \R $ (we will consider $t$ to be the vertical height in
$\esf^2(1)\times \R $); recall that $\varphi \in [0,\pi ]$ measures
the angle with the positive vertical axis in $\R^3$ and $\theta \in
[0,2\pi )$.
Let $W=H^+ -C_{\de} ^-\subset \R^3-\{ \vec{0}\} $ be the closure in $H^+$
of the convex region above $C_{\de }$, and let $\wt{W}=F(W)$ be the
closed solid cylinder bounded by $A_{C_{\de }}$;
 see Figure~\ref{figure2new}.
\begin{figure}
\begin{center}
\includegraphics[width=17.3cm,height=5cm]{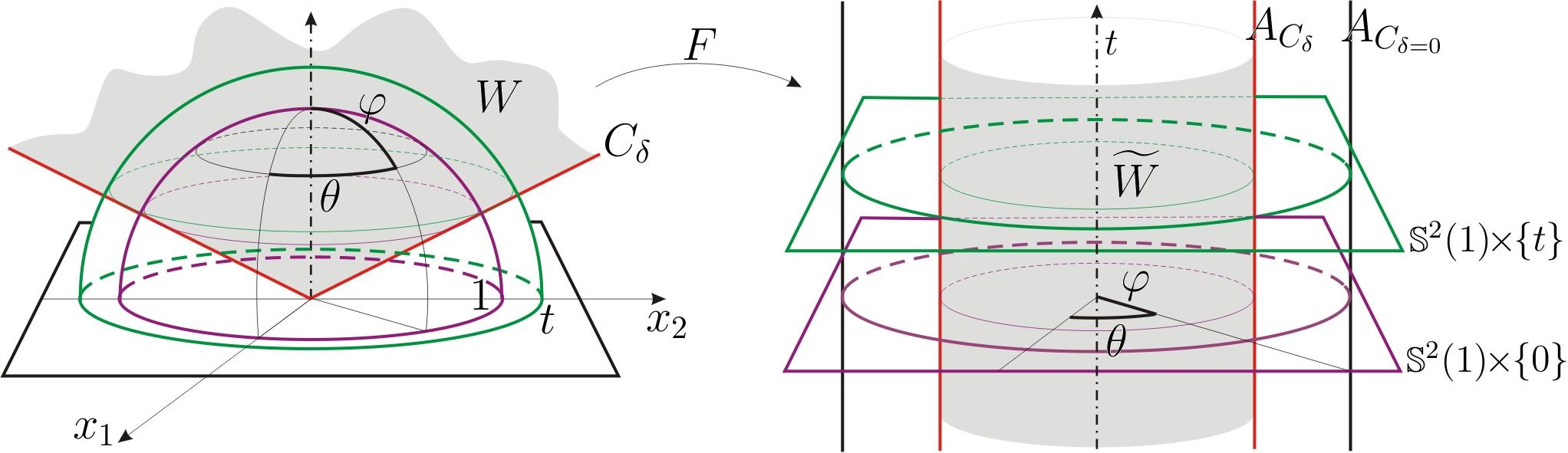}
\caption{$F(p)=(\frac{p}{R(p)}, \log R(p))$ is an isometry between
$(\R^3-\{ \vec{0}\} ,\widehat{g})$ and $(\esf^2(1)\times \R ,\wt{g})$.}
\label{figure2new}
\end{center}
\end{figure}

Consider a leaf $L$ of a minimal lamination $\mathcal{L}$ of $\R^3-\{ \vec{0}\} $
satisfying the
hypotheses and conclusions of Lemmas~\ref{lema2new} and~\ref{lema1new}.
Therefore, $\widehat{L}=(L,\widehat{g}|_L)\subset
(\R^3-\{ \vec{0}\} ,\widehat{g})$ is
complete.  Given a surface $\Sigma \subset \R^3-\{ \vec{0}\} $, we will
denote by $\widetilde{\Sigma }$ the image $F(\Sigma )$ under $F$.

\begin{lemma}
  \label{lema3new}
In the above situation, the $\widetilde{g}$-area of $(\widetilde{L}\cap \wt{W})\cap \{ T<t<T+1\} $ is
bounded independently of $T\in \R $. In particular, there exists $c>0$ such that
for any $a>1$, the $\widetilde{g}$-area of $(\widetilde{L}\cap \widetilde{W})\cap \{ T<t<T+a\} $
is not greater than $c\, a$.
\end{lemma}
\begin{proof}
  Suppose first that there exists a sequence $r_n>0$ such that the
area of $L\cap W\cap \{ r_n<R<e\cdot r_n\} $ is greater than any
constant times $r_n^2$ for $n$ large enough. Then the areas of the
surfaces $\frac{1}{r_n}\left( L\cap W\cap \{ r_n<R<e\cdot r_n\}
\right)$ become unbounded in a compact region of $H^+$. As $L$ is assumed
to have quadratic curvature decay, then it
follows that a subsequence of $\frac{1}{r_n}L$ converges to a
lamination of $\R^3-\{ \vec{0}\} $, with a leaf $L_1$ which is
either a limit leaf or a leaf for which the multiplicity of the convergence
of the sequence $\{ \frac{1}{r_n}L\} _n$ to $L_1$ is infinite.
In either of these two cases, $L_1\subset H^+$ which must be a horizontal plane;
see Corollaries~\ref{sc1} and \ref{sc2}. The existence of such
a plane implies that the curves of the set $\Lambda $ (defined in the
paragraph just before Section~\ref{subsect1}) are of type I, and that
below $C_\de$, $L$ consists of annular
graphs. This property implies that the lamination $\mathcal{L}$
of which $L$ is a leaf contains
an entire graph, which must be a horizontal plane of positive height. Applying
the arguments in the proof of Lemma~\ref{lema2new} we can rule out the
existence of such a plane, which then proves that the area of
$L\cap W\cap \{ r<R<e\cdot r\} $ divided by $r^2$ is bounded
from above independently of $r>0$.

The last property implies that there exists some constant $c_1>0$ such that
for any $\l >0$,
\[
r^{-2}\mbox{Area}\left[ \l (L\cap W)\cap \{ r<R<e\cdot r\} \right]=
(r/\l )^{-2}\mbox{Area}\left[ (L\cap W)\cap \{ \frac{r}{\l }<R<e\cdot \frac{r}{\l }\} \right]
\leq c_1,
\]
Since vertical translations $\psi _T(q,t)=(q,t+T)$ in $\esf^2(1)\times \R $ are isometries of the
product metric $\wt{g}$ on $\esf ^2(1)\times \R $,
then homotheties centered at the origin are isometries of $(\R^3-\{ 0\} ,\widehat{g})$. Thus,
\[
\mbox{Area}_{\wt{g}}\left[ (\widetilde{L}\cap \wt{W})\cap \{ T<t<T+1\} \right] =
\mbox{Area}_{\wt{g}}\left[ \psi _{-T}(\widetilde{L}\cap \wt{W})\cap \{ 0<t<1\} \right]
\]
\[
=\mbox{Area}_{\widehat{g}}\left[ \left( (e^{-T}L)\cap W\right) \cap \{ 1<R<e\} \right]
\stackrel{(A)}{\leq }\mbox{Area}_{g}\left[ (e^{-T}L)\cap W \cap \{ 1<R<e\} \right] \leq
c_1,
\]
where in (A) we have used that $\widehat{g}=\frac{1}{R^2}g$ and $R\geq 1$ in the range where we are computing areas.
This finishes the proof of the lemma.
\end{proof}

\begin{proposition}
  \label{propos1new}
The complete surface $\left( \wt{L}=F(L),\widetilde{g}|_{\wt{L}}\right) $
has at most quadratic area growth.
\end{proposition}
\begin{proof} We will divide the proof in two cases, according to the type of curves in $\Lambda $ (with
the notation just after Section~\ref{subsect1}).
\par
\vspace{.2cm}
{\sc Suppose that the curves in $\Lambda $ are of type~I.}
\par
\vspace{.2cm}
Given $\G \in \Lambda $, $\wt{\G }:=F(\G )$ is
an almost horizontal circle in $A_{C_\de}=F(C_{\de })=\partial \wt{W}$.
Let $\widetilde{\Lambda}$ denote the collection of these
curves $\widetilde{\G }$. Let $E(\G )$ be the component of $L\cap C_{\de }^-$ with
$\partial E(\G ) = \G$. Since each of the components in $L\cap
C_{\de '}$ is of the same type as $\G $ for any $\de '\in (0,\de )$,
then $E(\G )$ is an annulus. Let $\widetilde{E}(\G )=F(E(\G ))\subset \esf
^2(1)\times \R $ be the related surface. Note that
$\widetilde{E}(\G )$ is asymptotic to the end of $\esf^1(1)\times
[0,\infty )$ and that $\widetilde{E}(\G )$ is a small $\varphi $-graph
over its projection to $\esf ^1(1)\times \R $;
see Figure~\ref{figure3new}.
In fact, this $\varphi $-graph
has small $\wt{g}$-gradient as such a gradient measures the $\wt{g}$-angle
between the tangent space to $\wt{E}(\G)$ with the vertical cylinders in
$\esf^2(1)\times \R $, but this $\wt{g}$-angle coincides with the angle
between the tangent space to $E(\G)$ with the cones $C_{\de '}$, $\de '\in
(0,\de )$ (since the metric $\wh{g}$ is conformal to usual inner product in
$\R^3$), which can be made arbitrarily small by Lemma~\ref{lema1new}.
\begin{figure}
\begin{center}
\includegraphics[width=14.6cm,height=10cm]{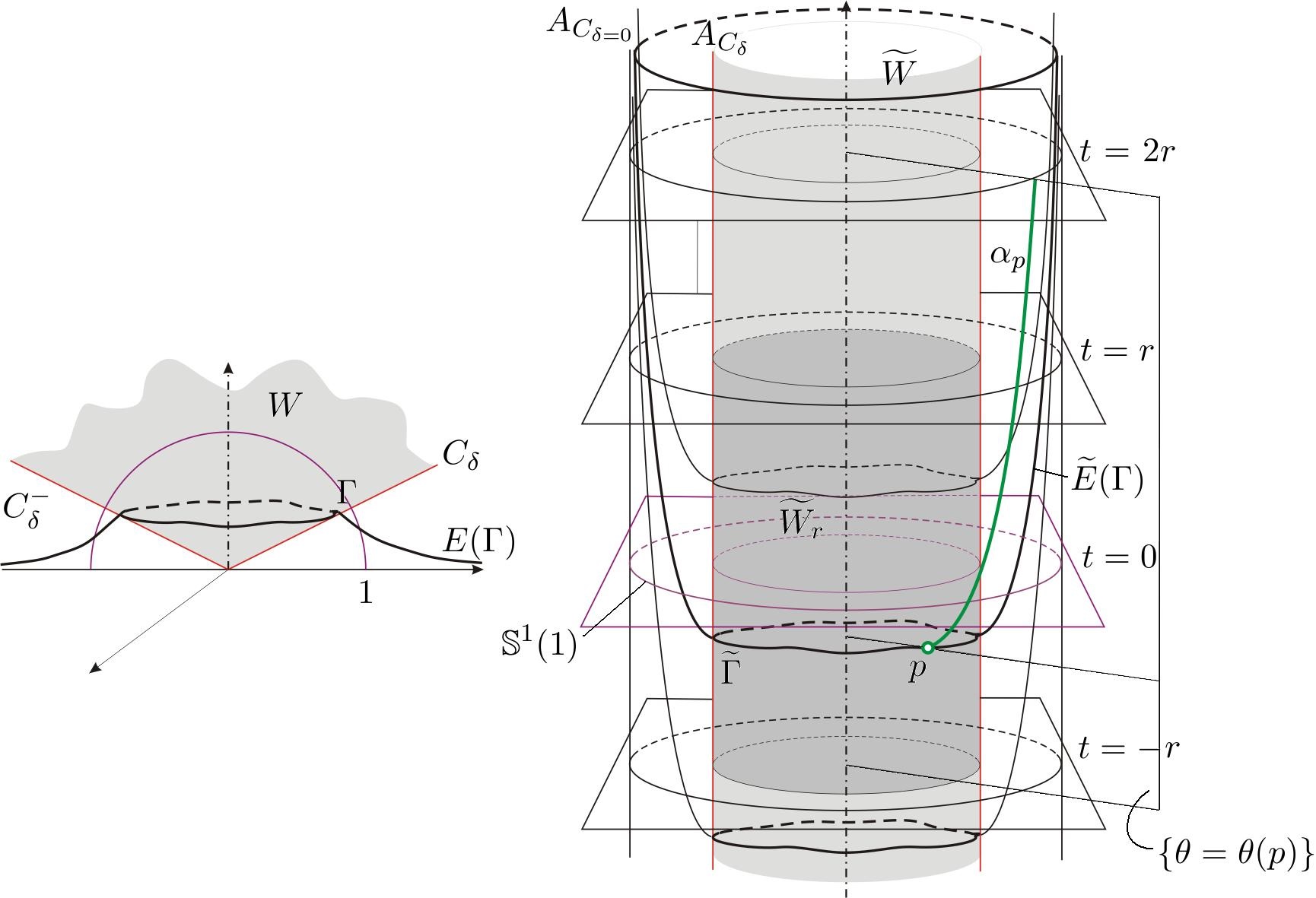}
\caption{Quadratic area growth in the case of curves in $\Lambda $ of type I.}
\label{figure3new}
\end{center}
\end{figure}

For $d>0$, let $[\widetilde{E}(\G )](d)$ be the subset of
$\widetilde{E}(\G )$ consisting of those points at intrinsic
distance at most $d$ from $\widetilde{\G} =
\partial \widetilde{E}(\G)$. Then for $\de
>0$ sufficiently small and $d\geq 1$, the $\widetilde{g}$-area of
$[\widetilde{E}(\G )](d)$ is less than $3\pi d$ and the limit as
$d \to \infty$ of such an area divided by $d$ is $2\pi $.

By Lemma~\ref{lema3new}, for $T_0>1$ fixed and for $T \in \R$, the
$\widetilde{g}$-area of $(\widetilde{L}\cap \wt{W})\cap \{ T-1<R<T+T_0+1\} $ is
bounded independently of $T$. Since the angle between $\partial
\wt{W}$ and $\widetilde{L} $ is small but bounded away from zero and the
components of $\widetilde{\Lambda}$
that intersect the region $\esf^2(1)\times [T, T+T_0] $ are contained in
$\esf^2(1)\times [T-\frac{1}{2}, T+T_0+\frac{1}{2}] $, then
the total $\widetilde{g}$-length of  components of $\widetilde{\Lambda}$
that intersect the region $\esf^2(1)\times [T, T+T_0] $ is bounded from above
independently of $T$ as is the number of  these components.

For $r>1$, let $\wt{W}_r=\wt{W}\cap (\esf  ^2(1)\times [-r,r])$ and let
$\widetilde{\Lambda}(r)$ denote the collection of those elements of
$\widetilde{\Lambda}$ which intersect $\partial \wt{W}_r$. For $r$
large, let $\Delta (r)\subset \widetilde{L}$ be union of
$\widetilde{L}\cap \wt{W}_r$ with the region
\[
V_r=\bigcup _{\widetilde{\G }\in \widetilde{\Lambda}(r)}\left(
\bigcup _{p\in  \widetilde{\G }} \a _p\right)
\]
where for each $p\in \widetilde{\G }$ with $\widetilde{\G } \in
\widetilde{\Lambda}(r) $,  (with related component $\G =F^{-1}(\wt{\G })\in \Lambda$),
$\a _p$ is the component of $\widetilde{E}(\G )\cap \{ \t =\t (p)\} $
whose end points are $p$ and a point at $\{ t=2r\} $, see Figure~\ref{figure3new} right.
Given $\widetilde{\G }\in \widetilde{\Lambda}(r)$,
the geometric meaning of $\bigcup _{p\in  \widetilde{\G }} \a _p$
is the compact annular piece of the
end $\widetilde{E}(\G )$ of $\widetilde{L}$ 
bounded by the almost horizontal circle $\widetilde{\G }$ 
and by a horizontal Jordan curve at constant height $2r$ in $\esf ^2(1)\times \R $
which is arbitrarily close to the circle $\esf ^1(1)\times \{ 2r\} $ if $r$ is taken large
enough; 
see Figure~\ref{figure3new} right. Note that $V_r$ is just
a finite union of these compact annuli, indexed by the curves
$\widetilde{\G }\in \widetilde{\Lambda}(r)$.
By the last paragraph,
there exists a universal constant $c>0$ such that for $r>1$, the
number of $\widetilde{\G} $-curves such that $\partial
\widetilde{E}(\G )$ intersects $\esf ^2(1)\times [-r,r]$, divided by
$r$, is less than $c$. Also, note that for $r$ large, the $\widetilde{g}$-length of any of these
$\widetilde{\G }$-curves is not greater than $3\pi $ and the $\widetilde{g}$-length of any of
the considered $\a _p$-curves is less than or equal to $4r$.
Therefore, the $\widetilde{g}$-area of $V_r$ is certainly less than
$3\pi \cdot 4r\cdot cr=12\pi cr^2$ for $r$
large. Since the $\widetilde{g}$-area of $\widetilde{L}\cap W_r$
grows linearly in $r$, then the $\widetilde{g}$-area of $\Delta (r)$ is at
most $13\pi cr^2$ for $r$ large.

Fix $p_0\in \widetilde{L}\cap W \cap (\esf ^2(1)\times \{ 0\} )$ and
let $B_{\widetilde{L}}(p_0,r)$ be the intrinsic open ball of radius
$r>0$ centered at $p_0$. We will show that
$B_{\widetilde{L}}(p_0,r)\subset \Delta (r)$. First note that
$B_{\widetilde{L}}(p_0,r)$ is contained in the region
$\esf^2(1)\times (-r,r)$. Let $p$ be a point in
$B_{\widetilde{L}}(p_0,r)$. If $p\in \wt{W}_r$, then by definition $p\in
\Delta (r)$. Suppose $p\notin \wt{W}_r$ and let $\g \subset \widetilde{L}$ be a
curve of $\widetilde{g}$-length less than $r$ joining $p$ with
$p_0$. Since $\partial \Delta (r)$ does not intersect the region
$\esf^2(1)\times (-r,r)$, then $\g$ does not intersect the boundary
of $\Delta (r)$.  As $p_0\in \Delta (r)$, we conclude $\g \subset \Delta (r)$. Hence,
$p\in \Delta (r)$. Finally, since $B_{\widetilde{L}}(p_0,r)\subset \Delta (r)$,
we deduce from the last paragraph that the intrinsic area growth of $\widetilde{L}$ is at
most quadratic. This completes the proof of Proposition~\ref{propos1new} provided that
the curves in $\Lambda $ are of type I.
 \par
\vspace{.2cm}
{\sc Next assume that the curves in $\Lambda $ are of type II.}
\par
\vspace{.2cm}
As in the previous case of type I curves, we consider the similar objects:
\begin{itemize}
  \item $\G \in \Lambda $, which is now a spiraling curve in $C_{\de }$ limiting down to
  $\vec{0}$ and rotating infinitely many times around $C_{\de }$. We have $k$ of these
curve components in $\Lambda =\{ \G _1,\ldots ,\G _k\} $.

\item $E(\G )$, the component of $L\cap C_{\de }^-$ with $\partial E(\G )=\G $. Note that
$E(\G )$ is topologically a half-plane.

\item $\widetilde{E}(\G )=F(E(\G ))$, which  is also topologically
a half-plane, and it is a small $\varphi $-multigraph with small gradient over the cylinder
$\esf ^1(1)\times \R $ (see the previous type I curve case for this argument).
Away from its boundary, $\widetilde{E}(\G )$
is asymptotic as a set to $\esf ^1(1)\times \R $.

\item $\wt{W}=F(W)$, an infinite open solid vertical cylinder in $\esf ^2(1)\times \R $
bounded by $A_{C_{\de }}=F(C_{\de })$. Given $r>0$, $\wt{W}_r=\wt{W}\cap \left( \esf^2(1)
\times [-r,r]\right) $ is the portion of $\wt{W}$ between heights $-r$ and $r$.

\item $\widetilde{\G }=\partial \widetilde{E}(\G )=F(\G )$, an infinite spiral in $A_{C_{\de }}$
whose height function is proper. We have $k$ of these helix-type curves, and $\widetilde{\Lambda }=
F(\Lambda )=\{ \widetilde{\G }_1,\ldots ,\widetilde{\G }_k\} $.
\end{itemize}

For $r>1$, define
$\widetilde{\Lambda}(r)$  to be the set of components in
$\bigcup _{i=1}^k\left[ \widetilde{\G}_i\cap W_r\right] $
which contain an end point at height
$r$ and the other end point at height $-r$. Since the curves in
$\widetilde{\Lambda}$ are very horizontal, for any $r>1$ the set
$\widetilde{\Lambda}(r)$ consists of $k$ spiraling arcs
$\widetilde{\G}_1 (r), \ldots, \widetilde{\G}_k (r)$.
Given $i\in \{ 1,\ldots ,k\} $ and a point $p$ in a spiraling arc
$\widetilde{\G }_i(2r)\in \widetilde{\Lambda }(2r)$, we call
$\a _p$ to the component of $\widetilde{E}(\G )\cap \{ \t =\t (p)\} $
 whose end points are $p$ and a point at $\{ t=2r\} $. Thus, $\a _p$ is
a planar Jordan arc and as we move the point $p$ along $\widetilde{\G }_i(2r)$
(with $i$ fixed), the arcs $\{ \a _p\ | \ p\in \widetilde{\G }_i(2r)\} $
define a foliation of the compact disk $\bigcup _{p\in \widetilde{\G }_i(2r)}\a _p
\subset \widetilde{E}(\G _i)$. Let
\begin{equation}
  \label{eq:Vr}
V_r=\bigcup _{i=1}^k\left( \bigcup _{p\in \widetilde{\G }_i(2r)}\a _p\right)
\end{equation}
be the union of these $k$ disks, and let $\Delta (r)\subset \widetilde{L}$ be union of
$\widetilde{L}\cap \wt{W}_{2r}$ with $V_r$; see Figure~\ref{Vr}.
\begin{figure}
\begin{center}
\includegraphics[width=15cm]{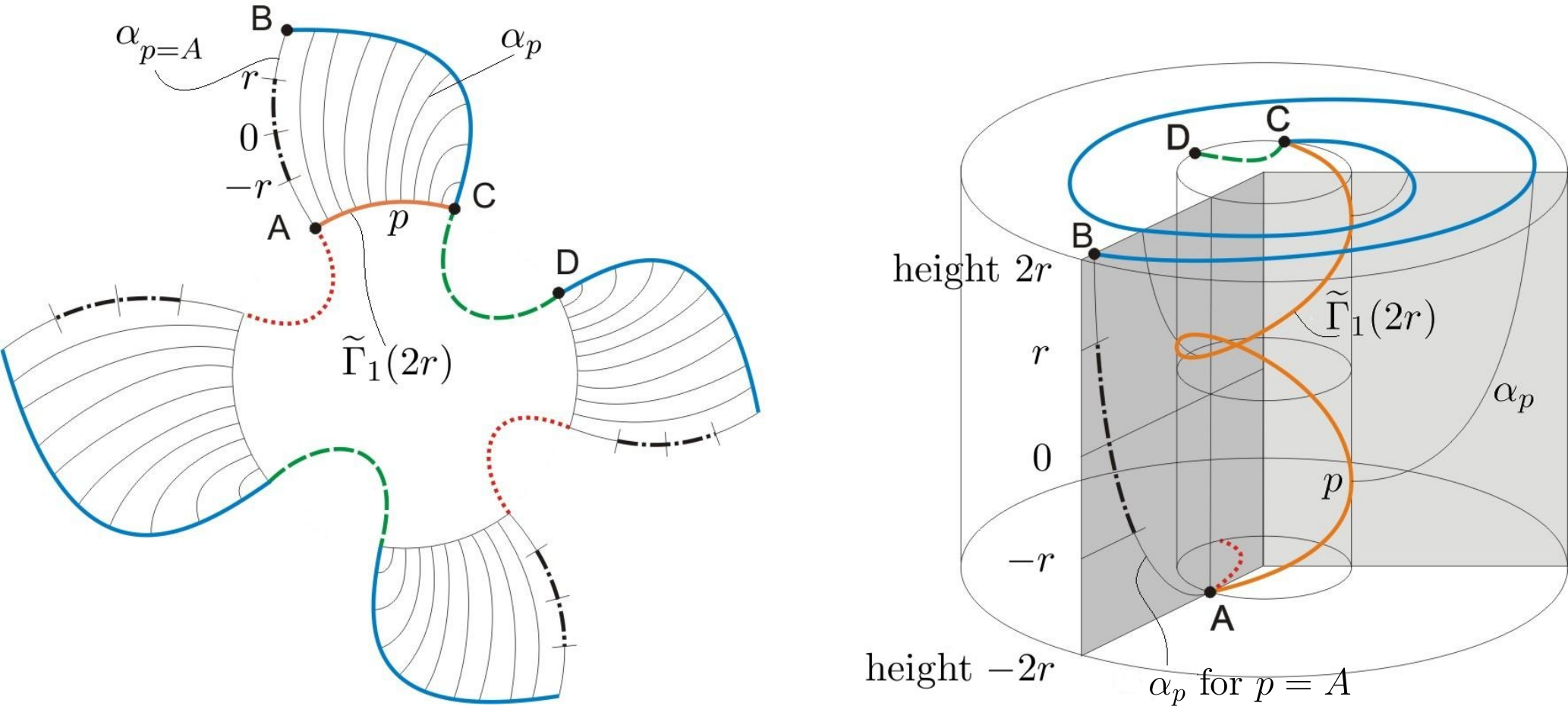}
\caption{Left: The cross-shaped region is an intrinsic representation of
$\Delta (r)$ if the curves in $\Lambda $ are of type II, here $k=4$. The
$k$ disks in the ``wings'' of $\Delta (r)$ form the set $V_r$, which is foliated by the Jordan arcs
$\a _p$. The remaining ``central'' portion $\widetilde{L}\cap \wt{W}_{2r}$ of $\Delta (r)$
might have positive genus. The boundary $\partial (\widetilde{L}\cap \wt{W}_{2r})$
of this central region is a cyclic union of spiraling arcs $\widetilde{\G }_i(2r)$
joined by arcs at constant
heights $2r$ (dashed in the figure) and $-2r$ (dotted in the figure).
Right: a schematic representation
of $\bigcup _{p\in  \widetilde{\G }_1(2r)} \a _p\subset V_r\subset \Delta (r)$
in $\esf^2(1)\times \R $ (the spiraling arc $\widetilde{\G }_1(2r)$
should be almost horizontal,
although we have represented
it reasonably steep in order to clarify the picture). The (blue) curve
joining $B$ to $C$ lies entirely at height $2r$ in $\esf ^2(1)\times \R $,
and spirals an arbitrarily large number of times (taking $r>0$ sufficiently large)
towards $\esf ^1(1)\times \{ r\} $ from its convex side.}
\label{Vr}
\end{center}
\end{figure}

The bound of $\wt{g}$-area density given by Lemma~\ref{lema3new} implies that
there exists a universal constant  $c>0$ such that for $r>1$ large,
the sum of the lengths of the curves in $\widetilde{\Lambda}(2r) $,
divided by $r$, is less than $c$. This last property together with the fact that
the $\widetilde{g}$-length of the $\a _p$-curves is not greater than $5r$ for
$\de >0$ sufficiently small, implies
that the $\widetilde{g}$-area of $V_r$ is less than or equal to $cr\cdot 5r=5cr^2$.
Since the $\widetilde{g}$-area of
$\widetilde{L}\cap \wt{W}_{2r}$ grows linearly in $r$, then we conclude that the
$\widetilde{g}$-area of $\Delta (r)$ is at most $6cr^2$ for $r$ large.

Let $p_0 \in \widetilde{L} \cap \wt{W}$ be a point at height $0$. It
remains to check that $B_{\widetilde{L}}(p_0,r)\subset \Delta (r)$ in
order to conclude that the intrinsic $\wt{g}$-area growth of $\widetilde{L}$
is at most quadratic in case of type II curves in $\Lambda $.
Observe that the boundary of $\Delta (r)$
intersects the region $\mbox{Ext}(\wt{W}_r)$ defined as the intersection of
the exterior of $\wt{W}$ with $\esf^2(1)\times [-r,r]$
along a finite collection $\Sigma=
\{\a_1, \ldots, \a_k \}$ of almost vertical arcs $\a_i\subset
 \a_{p_i}$, where $p_i$ is the extremum of $\widetilde{ \G} _i(2r)$
at height $t=-2r$ (these $\a _i$-curves are represented by dotted-dashed lines
$- \cdot - \cdot  - $ in Figure~\ref{Vr}).

By the previous argument for type~I curves,
in order to prove $B_{\widetilde{L}}(p_0,r)\subset \Delta (r)$ it
suffices to show that if $\g \subset \wt{L}$ is an arc
of $\widetilde{g}$-length less than $r$ starting at $p_0$, then
$\g \cap \partial \Delta (r)=\mbox{\O }$.
Observe that $\g $ cannot intersect the portion of
$\partial \Delta (r)$ at height $2r$ or $-2r$ (since
$p_0$ is at height zero and the $\wt{g}$-length of $\g $ is
less than $r$), so it suffices to discard any intersection of
$\g $ with a portion of $\partial \Delta (r)$ in $V_r$. So assume that the arc
$\g $ contains a point $q\in V_r\cap \partial \Delta (r)$. As
$\g $ has $\wt{g}$-length less than $r$ and starts at height zero,
then $\g $ lies entirely in $\esf^2(1)\times [-r,r]$. This implies
that $q$ must lie in one of the almost vertical arcs $\a _i$, $i=1,\ldots ,k$,
and by continuity $\g $ must intersect $\wt{\G }_i\cap (\esf^2(1)\times
[-r,r])$ at some point $q'$
so that the subarc $\g _{q,q'}$ of $\g $ between $q$ and $q'$ lies entirely in
the topological half plane $\widetilde{E}(\G _i)$.
Note that we can assign to the spiraling arc $\wt{\G }_i$ a well-defined,
real angle valued function $\widehat{\theta }$
which coincides with $\t $ mod $2\pi$. Furthermore, $\wt{\theta }$
extends to a continuous function (denoted in the same way) on $\widetilde{E}(\G)$.
Since the curve $\wt{\G }_i$ is almost horizontal and
$\widehat{\theta }\circ \g _{q,q'}$ is continuous, the absolute
difference between the $\widehat{\theta}$-values of the end points
of $\gamma _{q,q'}$ is much larger than $r$. This implies that the
$\wt{g}$-length of $\g _{q,q'}$ is also larger than $r$, which is
a contradiction. This contradiction demonstrates that
$B_{\widetilde{L}}(p_0,r)\subset \Delta (r)$, and thus,
the intrinsic area growth of $\widetilde{L}$ is at
most quadratic provided that the curves in $\Lambda $ are of type II.
Now the proof of Proposition~\ref{propos1new} is complete.
\end{proof}

We are now ready to reprove Proposition~\ref{propos1}.
Arguing by contradiction, suppose $L$ is a leaf of a minimal lamination
${\cal L}$ of $\R^3-\{ \vec{0}\} $ with quadratic decay of curvature,
and suppose that $L$ is not proper in $\R^3-\{ \vec{0}\} $.
By Lemmas~\ref{lema2new} and~\ref{lema1new},
after a rotation we can assume $L\subset H^+$, lim$(L)=\{ x_3=0\} -\{ \vec{0}\} $.
Consider the conformal change of metric
$\widehat{g}=\frac{1}{R^2}\langle ,\rangle $ on $\R^3-\{ \vec{0}\} $. Since the map $F$
defined in (\ref{mapF}) is an isometry, Proposition~\ref{propos1new} ensures
that the complete surface $\left( L,\widehat{g}|_{L}\right) $
has at most quadratic area growth. This property implies
recurrence for Brownian motion, see for instance
Grigor'yan~\cite{gri1}. Recurrence is a property that only depends on the conformal
class of $\widehat{g}|_L$, hence $L$ (with its original metric induced by the usual inner product
of $\R^3$) is also recurrent. The classical Liouville theorem for recurrent manifolds applies to
the positive harmonic coordinate function $x_3$ on $L$ and gives that $x_3$ is constant on $L$,
which is clearly a contradiction. This contradiction shows
that every leaf $L$ of $\mathcal{L}$ is properly embedded in $\R^3-
\{ \vec{0}\} $. The rest of the proof of the proposition is the same as the one
given at the end of Section~\ref{sec4}.

\center{William H. Meeks, III at profmeeks@gmail.com\\
Mathematics Department, University of Massachusetts, Amherst, MA 01003}
\center{Joaqu\'\i n P\'{e}rez at jperez@ugr.es \qquad\qquad Antonio Ros at aros@ugr.es\\
Department of Geometry and Topology, University of Granada, Granada, Spain}

\bibliographystyle{plain}
\bibliography{bill}
\end{document}